\documentclass[11pt]{amsart}
\usepackage{amsmath, amsfonts, stmaryrd,amssymb,mathtools,esint,mathrsfs}
\usepackage{amsthm}
\usepackage[utf8]{inputenc}
\usepackage[T1]{fontenc}
\usepackage{soul}
\usepackage[todonotes={textsize=scriptsize}]{changes}
\usepackage[colorlinks=true,linkcolor=blue, citecolor=red]{hyperref}
\usepackage{color}
\usepackage{fullpage}

\newtheorem{thm}{Theorem}[section]
\newtheorem{prop}{Proposition}[section]
\newtheorem{lem}{Lemma}[section]

\newtheorem{defi}{Definition}[section]
\newtheorem{ass}{Assumption}[section]

\newtheorem{rem}{Remark}[section]

\numberwithin{equation}{section}

\def\R{\mathbf{R}}
\def\P{\mathbf{P}}
\def\N{\mathbf{N}}

\def\Z{\textnormal{Z}}

\def\E{\mathrm{E}}
\def\dd{\mathrm{d}}

\def\B{\mathrm{B}}
\def\V{\mathrm{V}}
\def\W{\mathrm{W}}

\def\C{\mathrm{C}}
\def\L{\mathrm{L}}
\def\div{\mathrm{div}}
\def\Ll{\mathrm{L}_{\textnormal{loc}}}
\def\H{\mathrm{H}}
\def\K{\mathrm{K}}

\def\D{\mathrm{D}}
\def\V{\mathrm{V}}
\def\X{\mathrm{X}}
\def\Z{\mathrm{Z}}

\newcommand{\init}{\textnormal{in}}

\def\ep{\varepsilon}
\def\ffi{\varphi}

\newcommand{\conv}[2]{\operatorname*{\longrightarrow}_{#1 \rightarrow #2}}

\newcommand{\vertiii}[1]{{\vert\kern-0.08em\vert\kern-0.08em\vert #1 \vert\kern-0.08em\vert\kern-0.08em\vert}}

\title{On uniqueness for the three-dimensional Vlasov-Navier-Stokes system}
\author{D.~Han-Kwan}
\address{LMJL, CNRS, Nantes Université.}
\email{daniel.han-kwan@univ-nantes.fr}
\author{É.~Miot}
\address{IF, CNRS, Universit\'e Grenoble-Alpes.}
\email{evelyne.miot@univ-grenoble-alpes.fr}
\author{A.~Moussa}
\address{LJLL, Sorbonne Université, Université Paris Cité, CNRS, Inria.}
\email{ayman.moussa@sorbonne-universite.fr}
\author{I.~Moyano}
\address{LJAD, CNRS, Universite Côte-d'Azur.}
\email{Ivan.moyano@univ-cotedazur.fr}

\begin{document}
\maketitle

\begin{abstract}
We study the problem of uniqueness of Leray solutions to the  three-dimensional Vlasov-Navier-Stokes system. We establish uniqueness whenever the fluid velocity field belongs to the Cannone-Meyer-Planchon class, which allows to go beyond the Osgood uniqueness class. A stability estimate in this setting is also provided.
\end{abstract}
\section{Introduction}
\subsection{The Vlasov-Navier-Stokes sytem}
The Vlasov–Navier–Stokes (VNS) system arises as a kinetic-fluid coupling model describing the dynamics of dilute particles immersed in a viscous incompressible fluid. It plays a central role in the modeling of sprays and serves as a mathematically rich prototype combining aspects of kinetic theory and fluid dynamics. In the whole space $\R^3$, the system is the following:
\begin{equation}
    \text{(VNS)}\left\{\begin{aligned}
    &\partial_t u + u\cdot\nabla u - \Delta u +\nabla p = \int_{\R^3} f(v-u)\,\dd v,\\
    &\div u = 0, \\
    &\partial_t f+ v \cdot\nabla_x f + \nabla_v\cdot[f(u-v)] = 0.
    \end{aligned}\right.
\end{equation}
It consists of the incompressible  Navier–Stokes equations for the fluid velocity field $u:\R_+\times\R^3\rightarrow\R^3$ coupled with a Vlasov equation for the particle density function $f:\R_+\times\R^3\times\R^3\rightarrow\R_+$ in phase space, with interaction mediated by a drag acceleration for the kinetic component and the so-called \emph{Brinkman force} as a source term in the fluid equation.  Note that with $\mathbf{P}$ the Leray projector on solenoidal vector fields, the Navier-Stokes equations in (VNS) can be recast as
$$
\partial_t u + \mathbf{P}[u\cdot\nabla u] - \Delta u = \mathbf{P}\int_{\R^3} f(v-u)\,\dd v.
$$

\subsection{State of the art}
The existence of global weak solutions for the VNS system has been known since the early days of the theory \cite{anobou,bdgm}. The minimal setting allowing to exploit the energy dissipation of the system and thus to obtain the existence of global solutions is the Leray framework for the fluid component and the  renormalized setting (with Sobolev vector fields) for the kinetic one. This can be ensured provided the initial data satisfy some minimal regularity assumption, which motivates the following definition.
\begin{defi}\label{def:admi}
  We say that $(f^\init,u^\init)$ is an \textsf{admissible} initial data if:
  \begin{itemize}
  \item[$\bullet$] $u^\init \in\mathrm{L}^2(\R^3)$ is divergence-free, namely $\div u^\init=0$ ;
  \item[$\bullet$] $f^\init$ is non-negative and essentially bounded ;
  \item[$\bullet$] $(x,v)\mapsto (1+|v|^2)f^\init(x,v)$ is integrable.
  \end{itemize}
\end{defi}
The moments of the kinetic component will play a crucial role in our analysis, and we will often ask some decay in the velocity variable for the density function. The following notations will therefore be useful (we omit here the time variable). For $g:\R^3\times\R^3\rightarrow\R_+$ we use the notation $m_k g$ (for $k\in\N$) its (scalar) velocity moments that is
\[\forall x\in\R^3,\qquad m_k g(x) := \int_{\R^3} g(x,v)|v|^k\,\dd v,\]
and we write
\[ M_k g  := \int_{\R^3} m_k g(x) \,\dd x = \iint_{\R^3\times\R^3} g(x,v) |v|^k\,\dd v\,\dd x.\]
We also introduce the macroscopic density function defined by
$$\forall x\in\R^3,\qquad \rho_g(x) = \int_{\R^3} g(x,v)\,\dd v$$ and we denote by $j_g$ the current defined by
\begin{align*}
\forall x\in\R^3, \qquad j_g(x) = \int_{\R^3} g(x,v) v\,\dd v.
\end{align*}
Lastly, for any (sufficiently smooth) vector-field $w:\R^3\rightarrow\R^3$ we define the following functionals (energy and dissipation) 
\begin{align}
      \label{def:E}    \textnormal{E}(g,w) &= \frac12 \int_{\R^3} |w(x)|^2\,\dd x +\frac12  \iint_{\R^3\times\R^3} g(x,v)|v|^2\,\dd x\,\dd v,\\
       \label{def:D}\textnormal{D}(g,w) &= \int_{\R^3} |\nabla w(x)|^2\,\dd x + \iint_{\R^3\times\R^3} g(x,v)|v-w(x)|^2\,\dd x\,\dd v.
  \end{align}
Following \cite{hkm3} we define \emph{Leray solutions} for the VNS system.
\begin{defi}\label{def:leray}
  For any admissible data $(f^\init,u^\init)$ in the sense of Definition~\ref{def:admi}, a \textsf{Leray solution} of the \textnormal{VNS} system is a pair $(f,u)$, with a fluid component $u\in\Ll^\infty(\R_+;\mathrm{L}^1(\R^3)) \cap \Ll^2(\R_+;\H^1(\R^3))$, a kinetic component $f$ essentially bounded, such that $t\mapsto M_0 f(t) + M_2 f(t)$ belongs to $\Ll^\infty(\R_+)$, solving the \textnormal{VNS} system in the distributional sense and satisfying the estimate 
  \begin{align}\label{ineq:nrj}
      \textnormal{E}(f(t),u(t)) + \int_0^t \textnormal{D}(f(s),u(s))\,\dd s \leq \E^\init:=\textnormal{E}(f^\init,u^\init),
  \end{align}
  where the energy and dissipation functionals $\E$ and $\D$ are defined in \eqref{def:E} -- \eqref{def:D}.
\end{defi}
For admissible data, existence of a Leray solution is known for various settings and boundary conditions  \cite{anobou,bdgm,hkm3,bgm}. The aim of this article is to provide a weak-strong uniqueness result (see Subsection~\ref{subsec:main} for  precise statements) for the VNS system, in the three-dimensional setting. This is a natural extension of our previous contribution \cite{hkm3} in which uniqueness of Leray solutions for the VNS system was established in the case of dimension $2$. Namely, the main result of \cite{hkm3} is
\begin{thm}
  In dimension $2$, consider an admissible initial data $(u_0,f_0)$ such that $f_0(x,v)\lesssim (1+|v|)^{-q}$ for some $q>4$. The Vlasov-Navier-Stokes system admits a unique Leray solution associated with $(u_0,f_0)$.
\end{thm}
In this bidimensional setting, uniqueness was already known for each equation, Vlasov or Navier-Stokes, taken independently from the other. Indeed, on one hand, given a Sobolev vector field, uniqueness of renormalized solutions for the corresponding linear transport equation is ensured by \cite{dipernalions} and on the other hand, uniqueness for 2D Navier-Stokes in ensured by the use of Ladyzhenskaya's inequality, a result  attributed to Lions and Prodi \cite{prodilions}. Of course, when studying uniqueness for the VNS system, the difficulty stems from the nonlinear coupling between the two phases. This situation is quite reminiscent for instance to the one of the Vlasov-Poisson equation (in which uniqueness is even more straightforward when both components are taken independently) and actually the core functional used in \cite{hkm3} was inspired by the one used in Loeper's uniqueness result for the Vlasov-Poisson system \cite{loeper}. 

Now, what about dimension three? It is important to note that any uniqueness result for the VNS system also implies uniqueness for the Navier-Stokes equation (taking $(0,u^\init)$ as initial data). Therefore, positive or negative results for the sole Navier-Stokes equation are an expected barrier. In dimension $3$ the situation is thus by far more complex, and the uniqueness or non-uniqueness of Leray solutions is unknown at the time of writing of this paper. The  work \cite{colomboetal} recently proved the non-uniqueness of Leray solutions for a {\it forced} Navier-Stokes equation, giving a strong hint that uniqueness may not hold in general; moreover let us remark that the Vlasov-Navier-Stokes system itself can be seen as a forced Navier-Stokes equation, which clearly suggests that uniqueness of Leray equations for VNS cannot be expected. See also the  works \cite{buckmastervicol} and \cite{cheskidovluo} for what concerns non-uniqueness of {\it weak} solutions (in the distributional sense) for the (unforced) Navier-Stokes equation. 

Alternatively,  well-posedness in critical spaces has been put forward, leading to solutions that are named after their authors, \emph{e.g.} Fujita-Kato  \cite{fk}, Kato \cite{kato}, Cannone-Meyer-Planchon  \cite{cmp}, Koch-Tataru  \cite{kt}:
$$
\underbrace{\dot{\H}^{1/2} (\R^3)}_{\text{Fujita-Kato}} \hookrightarrow \underbrace{\mathrm{L}^3(\R^3)}_{\text{Kato}} \hookrightarrow \underbrace{\dot{\B}_{p,\infty}^{-1+3/p}(\R^3)}_{\substack{\text{Cannone-Meyer-Planchon}\\ \text{for }p\in [3,\infty)}} \hookrightarrow \underbrace{\mathrm{bmo}^{-1}}_{\text{Koch-Tataru}}
$$
This is a sequence of strict continuous embeddings. The Besov space $\dot{\B}_{p,\infty}^{-1+3/p}(\R^3)$ (for $p\in (3,\infty)$) will play an important role in this work; we do not wish to provide a precise definition in this introduction (we rather send the interested reader to Appendix~\ref{sec:lp}). Let us only note that for $p>3$, this is a space of distributions of possibly negative regularity, which contains $\mathrm{L}^3(\R^3)$.
In these critical frameworks, local-in-time solutions can be ensured to be global at the price of a smallness assumption on the data. For a comprehensive description of this large piece of literature, see e.g. \cite{LR2002} or \cite{bcd}. Of course, these functional settings differ from the one of Leray solutions and uniqueness is a priori only ensured in the subclass of such solutions. A substantial amount of articles tried to bridge these two points of view through \emph{weak-strong} uniqueness result, ensuring that whenever a strong enough solution exists, all Leray solution must coincide with this one, see \cite{serrin,cheminJAM,chemincpam,barker}. 
In particular the work \cite{chemincpam} by Chemin is a strong inspiration of our work. 
\medskip

As far as the VNS system is concerned, the existence theory in a smoother setting than Leray solutions, even though a natural question, has been poorly explored \emph{per se} in the early days of the theory. This question has been indirectly raised through the pionneering work of Choi and Kwon \cite{ck} in which a monokinetic behavior for the kinetic component in large time was proven, under the conditional assumption that the local density $\rho_f$ of the kinetic phase belongs to $\mathrm{L}^\infty(\R_+;\mathrm{L}^{3/2})$. It is while trying to produce  such type of solution to the VNS system that the first example of global strong solutions (under a smallness assumption for the initial data) appeared in the literature in \cite{hkm2} (still, not in the whole space but in the torus). This construction, which used the aforementioned \cite{ck}, led to the first rigorous proof of monokinetic behavior for the system. Shortly after, similar results were established in different geometric settings (\cite{hk} for the whole space, \cite{ehkm} for a bounded domain). More recently, Danchin revisited in \cite{Danchin} the whole space case $\R^3$, establishing optimal convergence rate and in passing a well-posedness result in the Fujita-Kato setting. Lastly, Danchin and Shou, in the two-dimensional torus, managed to prove in \cite{DS} the same large-time behavior but for initial data not necessarily close to equibrium.
\subsection{Main results}\label{subsec:main}
Before stating our main results, we need to introduce a few definitions and notations. A \emph{smooth dyadic approximate identity} is any sequence $(\psi_j)_{j\geq 0}$ originating from a smooth Schwartz function $\psi$ with integral $1$, that is: for all $j\in\N$ and all $x\in\R^3$ there holds $\psi_j(x) := 8^{j}\psi(2^j x)$. The following notion of \emph{well-approximated} functions is strongly inspired from \cite{chemincpam}, even though this exact definition was not explicitly highlighted. It turns out this well-approximation property is at the core of our uniqueness and stability results. 
\begin{defi}\label{def:well}
  Fix $T>0$. We say that a vector-field $u\in\mathrm{L}^2(0,T;\mathrm{L}^2(\R^3))$ is \textsf{well-approximated} if there exists a smooth dyadic approximate identity $(\psi_j)_j$ for which
  \begin{itemize}
\item[$\bullet$] There exists a positive sequence $(a_k)_k\in \textnormal{c}_0(\N)$ such that, for all $j\in\mathbf{N}$ we have \[\int_0^T \|u \star \psi_j\|_{\infty}^2\,\dd s + \int_0^T \|\nabla (u \star \psi_j)\|_{\infty}\,\dd s \leq  \sum_{k=0}^j a_k.\] 
\item[$\bullet$] There exists $\alpha>0$ such that 
 \[\int_0^T \|(u\star \psi_j)\otimes (u \star \psi_j) - (u\otimes u)\star \psi_j\|_2^2\,\dd s \lesssim 2^{-\alpha j}.\]
\end{itemize}
\end{defi}
{The well-approximation property automatically holds for smooth vector fields. The first condition is for instance trivially matched as soon as $u \in \mathrm{L}^1(0,T;\textnormal{W}^{1,\infty}(\R^3))$ because in that case, the l.h.s. is uniformly bounded, but this condition also allows vector fields which are far from being Lipschitz but for which the l.h.s. "does not blow-up too fast" with $j$. Likewise, the second condition is met when for instance $u \in \mathrm{L}^1(0,T;\mathrm{L}^\infty(\R^3))$ and both $u$ and $u\otimes u$ belong to a Sobolev space $\mathrm{L}^1(0,T;\H^s(\R^3))$ for some $s>0$ (because $\|f-f\star \psi_j\|_2 \lesssim 2^{-sj}$ for $f$ in $\H^s(\R^3)$). Again, the required regularity can actually be strongly relaxed, as we shall see.}

In order to quantify the decay in velocity of the kinetic component of the initial data, the following notation will be useful. For any initial data $f^\init$ and any $q\geq 0$, set
\begin{align}\label{def:Nq}
N_q(f^\init) := \sup_{(x,v)\in\R^3\times\R^3} (1+|v|^q) f^\init(x,v).
\end{align}

\subsubsection{Weak-strong uniqueness}

We are now in position to state our first main result. 
\begin{thm}\label{thm:univns}
  Fix $(f^\init,u^\init)$ an admissible initial data such that furthermore $u^\init$ belongs to $\H^s(\R^3)$ for some $s>0$ and $M_6 f^\init <+\infty$,  $N_q(f^\init)<+\infty$ for some $q>4$. For two Leray solutions $(f_1,u_1)$ and $(f_2,u_2)$ of the \textup{VNS} system on $[0,T]$, if either $u_1$ or $u_2$ is well-approximated in the sense of Definition~\ref{def:well}, then $(f_1,u_1)=(f_2,u_2)$.
\end{thm}

In Section~\ref{sec:besapp}, we will exhibit (following Chemin in \cite{chemincpam}) a sufficient condition for the well-approximation in terms of Besov regularity. In the present work, we also build a local solution in the corresponding setting (see Appendix~\ref{subsec:chemlen} for a precise definition of the Besov spaces $\dot{\B}_{p,\infty}^{-1+3/p}(\R^3)$ and  $\dot{\mathbb{B}}_p(T)$), ensuring  the non-emptiness of Theorem~\ref{thm:univns}. 

\begin{thm}
 \label{thm:BesovVNS}
  Fix $(f^\init,u^\init)$ an admissible initial data such that furthermore $u^\init$ belongs to $\H^s(\R^3)$ for some $s>0$ and $M_6 f^\init <+\infty$. Assume also for some $p\in(3,\infty)$ that $u^\init \in \dot{\mathbb{B}}_{p,\infty}^{-1+3/p}(\R^3)$ and for some $q>5$ that $N_q(f^\init)<+\infty$. Then, there exists a non-trivial interval of time $[0,T]$ on which all Leray solutions to the \textup{VNS} system coincide. Moreover, the fluid velocity of this unique solution belongs to $\dot{\mathbb{B}}_p(T)$.
\end{thm}
 Let us mention that in the recent work \cite{Danchin}, a uniqueness result is obtained in the framework of Fujita-Kato solutions, i.e. for initial fluid velocity  $u^\init \in \dot{\H}^{1/2}(\R^3)$.
This result builds on the method devised in \cite{hkm3} for the 2D case.
As a matter of fact, thanks to this method, one can check that uniqueness holds as soon as one Leray solution enjoys the following regularity:
\begin{equation}
\label{eq:strong}
u \in \mathrm{L}^2(0,T; \mathrm{L}^\infty) \cap \mathrm{L}^1 (0,T; \text{Log-Lipschitz}).
\end{equation}
This regularity can be ensured in the context of Fujita-Kato solutions (see \cite{cl}), hence uniqueness.


The functional setting of Theorem~\ref{thm:BesovVNS} corresponds to the celebrated Cannone-Meyer-Planchon framework, developed for the Navier-Stokes system. In particular it generalizes the aforementioned uniqueness result in the Fujita-Kato class. In the Cannone-Meyer-Planchon setting, the regularity~\eqref{eq:strong} cannot be ensured. Even worse, the velocity field cannot be ensured to belong to a space for which Osgood's uniqueness theorem holds. Consequently, the method of \cite{hkm3} must be substantially modified. The idea developed in the present paper, strongly inspired by Chemin's uniqueness result for Navier-Stokes \cite{chemincpam}, is to use an approximation argument in order to reduce to smooth fluid velocity fields, and finally to rely on the \emph{well-approximated} property of Definition~\ref{def:well} to ensure that the errors made are acceptable and mitigated in the limit.
As a result, for uniqueness to hold for the Vlasov-Navier-Stokes system, it is possible to go beyond Osgood's class, in sharp contrast with other nonlinear Vlasov equations: in particular for  the Vlasov-Poisson system, all known uniqueness results \cite{loeper,Miot,hmiot,ciss} involve a force field that is regular enough to apply a (second-order) Osgood theorem for the associated trajectories.

Relying on the long time description of smooth solutions that was uncovered in \cite{Danchin} (see also \cite{hk}), we obtain a global existence and uniqueness result for VNS in the Cannone-Meyer-Planchon setting.
\begin{thm}\label{thm:BesovVNS:glob}
Consider the assumptions of Theorem~\ref{thm:BesovVNS}. 
Assume furthermore that $u^\init \in \dot{\B}^{-3/2}_{2,\infty}(\R^3)$.
There exists a continuous decreasing function $\Phi$ such that the following holds. If 
\begin{equation}
\label{hyp:smallnessBesov}
\| u^\init\|_{\dot{\B}_{p,\infty}^{-1+3/p}(\R^3)}  + \E^\init \leq \Phi\big(  \| f^\init\|_{1} +  N_q(f^\init) +\|u^\init\|_{\dot{\B}_{2,\infty}^{-3/2}(\R^3)}\big),
\end{equation}
then all Leray solutions to the  \textup{VNS} system coincide on $\R_+$ and the fluid velocity of this unique solution belongs to $\dot{\mathbb{B}}_p(+\infty)$.
\end{thm}

\subsubsection{Stability}

Lastly, we propose a quantitative formulation of our Theorem~\ref{thm:univns} by means of a stability estimate. Of course, such estimate implies uniqueness but this result is obtained at the cost of an extra condition on the initial data (decay in space and velocity variables).

Below (see Theorem~\ref{thm:stabvns}), we will need to evaluate the distance between two initial (kinetic) data and this will be done using the following norm.

\begin{defi}\label{def:NR}
Fix $T>0$ and $R>0$. Consider
$$
\B_R := \Big\{ u \in \mathrm{L}^1(0,T;\mathrm{L}^\infty(\R^3)) \cap \mathrm{L}^2(0,T;\H^1(\R^3)),\, \div\, u=0, \, \|u \|_{\mathrm{L}^1(0,T;\mathrm{L}^\infty(\R^3))}\leq R\Big\}.
$$
For $f^\init\in \mathrm{L}^1(\R^3\times \R^3)$ such that $|v|f^\init\in \mathrm{L}^1(\R^3\times \R^3)$ and  $N_q(f^\init)<+\infty$ for some \textcolor{red}{$q>5$}, we introduce 
  \begin{align*}
    \mathcal{N}_{T,R}(f^\init) := \sup_{u\in\B_R} \sup_{t\in[0,T]} \Big\{\|m_0 |f_u|(t)\|_{\mathrm{L}^1(\R^3)\cap\mathrm{L}^\infty(\R^3)} + \|m_1|f_u|(t)\|_{\mathrm{L}^1(\R^3)\cap\mathrm{L}^\infty(\R^3)}\Big\},
  \end{align*}
where $f_u(t): = \Z_u(t)\# \overline{f}$ is the push-foward measure of $f^\init$ by the Lagrangian flow $\Z_u$ defined by the vector field $b(t,x,v):=-(v,u(t,x)-v)$.
\end{defi}
We will show in an appendix (see Lemma~\ref{lem:NR}) that Definition~\ref{def:NR} indeed makes sense in the context of Leray solutions to the VNS system (the quantity $\mathcal{N}_{T,R}(f^\init)$ is actually well-defined under the weaker assumption that $q>4$).
In this very last setting, note that $f_u$ is then the unique solution in $\mathrm{L}^\infty(0,T;\mathrm{L}^\infty(\R^3\times\R^3))$ \added{of}
\begin{equation}\label{eq:transport-vlasov}\partial_t f+ v \cdot\nabla_x f + \nabla_v\cdot[f(u-v)] = 0,\end{equation}
with initial value $f^\init$ at $t=0,$ is a consequence of the Diperna-Lions theory. More precisely, since $u \in \B_R$, setting $b(t,x,v)=-(v,u(t,x)-v)$ and $c(t,x)=3$ in \cite[Th. II.2]{dipernalions}, we observe that $b$ and $c$ satisfy the uniqueness assumptions of that theorem because $(t,x,v)\mapsto b(t,x,v)(1+|(x,v)|)^{-1}$ belongs to $\mathrm{L}^1(0,T;\mathrm{L}^\infty(\R^3\times\R^3))$.

Moreover, the Diperna-Lions theory also ensures in \cite[Th. III.2]{dipernalions} the existence and uniqueness of a Lagrangian flow $\Z_u(s,t,x,v)$ associated with $b$. In particular, writing $\Z_u=(\X_u,\V_u)$ we have for $a.e. s\in [0,T]$
\begin{equation}
\label{eq:ODEs}
\left\{
\begin{aligned}
&\frac{\dd}{\dd s}\X_u(s,t,x,v) = \V_u(s,t,x,v), \qquad \X_u(t,t,x,v)=x, \\
&\frac{\dd}{\dd s}\V_u(s,t,x,v) = u(s,x,v)- \V_u(s,t,x,v), \qquad \V_u(t,t,x,v)=v.
\end{aligned}
\right.
\end{equation}
Finally, by \cite[Th. III.2]{dipernalions}, the unique solution in $\mathrm{L}^\infty(0,T;\mathrm{L}^\infty(\R^3\times \R^3))$ (in the sense of distributions) to the transport equation
$$\partial_t g+ (v,u(x)-v) \cdot\nabla_{x,v} g = 0,\quad g(0)=f^\init$$
is given by $g(t,x,v)=f^\init(\Z_u(0,t,x,v)$. Setting $f_u(t)=e^{3t}g(t)$, it is clear that $f_u$ is a distributional solution in $\mathrm{L}^\infty(0,T;\mathrm{L}^\infty(\R^3\times \R^3))$ to \eqref{eq:transport-vlasov} with $f_u(0)=f^\init$, which is unique as recalled above. Thus it follows in particular that
$$f_u(t) = \Z_u(t,0) \# f^\init,
$$ since the right-hand side is also a distributional solution to \eqref{eq:transport-vlasov}.
In the following, the flow $\Z_u := (\X_u,\V_u)$ will often be referred to as the characteristic curves associated with $u$. 

\begin{thm}\label{thm:stabvns}
Consider $(f^\init_1,u^\init_1)$ and $(f^\init_2,u^\init_2)$ two admissible initial data such that $M_6 f^\init_k < +\infty$ and $N_q(f^\init_k)<+\infty$ for some $q>5$ and $k=1,2$. Assume furthermore $u^\init_1$ belongs to $\H^s(\R^3)$ for some $s>0$
. 
There exists a continuous function $\Psi$ such that the following holds. 
For $T>0$, assume that a Leray solution $(f_1,u_1)$ associated with $(f^\init_1,u^\init_1)$ such that $u_1$ is well-approximated, exists.
For any $\ep \in (0,1)$, there exists $\textnormal{C}>0$ such that, if
$$
\|u_1^\init-u_2^\init\|_2 + \mathcal{N}_{T,R}(f^\init_1-f^\init_2) \leq \Psi^{1/\ep}(T, u_1, (\E^\init_{k})_{k=1,2}, (M_6 f^\init_{k})_{k=1,2}, (N_q(f^\init_k))_{k=1,2}),
$$
the following estimate holds for any Leray solution $(f_2,u_2)$ associated with $(f^\init_2,u^\init_2)$:
  \begin{align*}
    \|u_1(t)-u_2(t)\|_2^2 + \int_0^t \|\nabla (u_1-u_2)(s)\|_2^2\,\dd s    \leq \textnormal{C} \Big[\|u_1^\init-u_2^\init\|_2^{2} + \mathcal{N}_{T,R}(f^\init_1-f^\init_2)^{2}\Big]^{1-\ep}.
  \end{align*}
  If furthermore $f^\init_1$ and $f^\init_2$ share the same mass, then we have (for a possibly different constant $\textnormal{C}$)
  \begin{multline*}
\W_1(f_1(t),f_2(t))^2+    \|u_1(t)-u_2(t)\|_2^2 + \int_0^t \|\nabla (u_1-u_2)(s)\|_2^2\,\dd s\\
    \leq \W_1(f^\init_1,f^\init_2)^2+ \textnormal{C} \Big[\|u_1^\init-u_2^\init\|_2^{2} + \mathcal{N}_{T,R}(f^\init_1-f^\init_2)^{2}\Big]^{1-\ep},
  \end{multline*}
  where $\W_1$ denotes the 1-Wasserstein distance. 
\end{thm}

\begin{rem}
This statement can be understood as a local stability estimate, in the sense that the initial data must be close enough.
  The closedness condition  can actually be made more precise from a view of the proof: 
  \begin{itemize}
      \item the dependence on $u_1$ is only through its well-approximated property, precisely through the sequence $(a_k)_{k\in \N}$ that appears therein;
      \item in particular, when $u_1$ belongs to the Besov space $\dot{\mathbb{B}}_p(T)$ (this implies the well-approximation property, see Lemma~\ref{prop:Besovwell}) the dependence on $u_1$ is only through its $\dot{\mathbb{B}}_p(T)$ norm.
      \item for any $T>0$ and $u_1$, we have
      $\Psi(T,u_1,0)=0$.
  \end{itemize}
  
\end{rem}

\subsection{Plan of the paper}

This paper is structured as follows. In Section~\ref{sec:uncoupled}, we study the evolution of the two key functionals allowing to evaluate the distance between two solutions, as introduced in the 2D case \cite{hkm3}. We argue as if the Navier-Stokes part and the Vlasov part were independent,  assume that one of the two solutions is attached with a smooth fluid velocity field (see in particular the upcoming Assumptions~\ref{ass:3}--\ref{ass:4}), but also allow for error terms.

Section~\ref{sec:proof1} is then dedicated to the proof of our uniqueness result, namely Theorem~\ref{thm:univns}. We consider a  sequence of smooth approximations of one of the two fluid velocity fields and thus are able to leverage the stability estimates from Section~\ref{sec:uncoupled}. All extra terms due to the approximation are considered as error terms. The whole point is to ensure that this procedure indeed converges and this is where the well-approximated property (recall Definition~\ref{def:well}) comes into play. This argument is strongly inspired by the work of Chemin \cite{chemincpam} for the Navier-Stokes equation. The subsequent Section~\ref{sec:proof4} provides a proof of Theorem~\ref{thm:stabvns}, that is the stability result for the VNS system, taking exactly the same route.

Next, Sections~\ref{sec:proof2} and~\ref{sec:proof3} are concerned with the application of our results to fluid data in Besov spaces.
We first apply our uniqueness result to prove local well-posedness in  Cannone-Meyer-Planchon spaces, which corresponds to Theorem~\ref{thm:BesovVNS}. In a second time, relying on the recent results of Danchin~\cite{Danchin}, we explain how small initial data give rise to global solutions, yielding a proof for Theorem~\ref{thm:BesovVNS:glob}.

The paper ends with three appendices.  In Appendix~\ref{sec:lp}, we gather useful elements related to the Besov spaces considered in this work. In particular, we explain in Section~\ref{sec:besapp}
 why the vector fields in the space $\dot{\mathbb{B}}_p(T)$ are indeed well-approximated. We collected in Appendix~\ref{sec:leray+} improved regularity properties satisfied by Leray solutions, as soon as the kinetic initial condition enjoys extra decay with respect to velocity. Lastly in Appendix~\ref{sec:contwas} we explain how  Loeper's functional can control the Wasserstein distance between two measures provided one of the two flows is regular enough.

\bigskip

\noindent {\bf Acknowledgements.} We warmly thank Isabelle Gallagher for inspiring discussion.

\section{Uncoupled stability estimates}
\label{sec:uncoupled}

\subsection{Assumptions and statements}
The proof of Theorem~\ref{thm:univns} relies on two independent building blocks, which are the objects of Lemmata \ref{lem:dynw} and \ref{lem:dynQ} below. Each of them evaluate the distance between two pairs $(\widetilde{f},\widetilde{u})$ and $(f,u)$, focusing on only one equation of the VNS system (the "fluid" part for Lemma \ref{lem:dynw} and the "kinetic" part for Lemma \ref{lem:dynQ}). We thus call these results \emph{uncoupled stability estimates}. 

In all this paragraph, the functions $\widetilde{f}$ and $f$ will always originate from the same initial data $f^\init$ through the push-forward of some (Lagrangian) flows $\widetilde{\Z},\Z: \R_+\times\R^3\times\R^3\rightarrow\R^3\times\R^3$, that is to say $$\widetilde{f}(t) =  \widetilde{\Z}(t)\# f^\init,\quad f(t) = \Z(t) \# f^\init.$$ 
In order to ease the presentation we will simply write $\rho$ and $j$ for the two first moments $\rho_f$ and $j_f$, and similary $\widetilde{\rho}$ and $\widetilde{j}$. Lastly, just as in \cite{hkm3} it will be convenient to measure the distance between these two kinetic components through the following functional
\begin{align}\label{def:Q}
Q_{\widetilde{\Z},\Z}(t) := \iint_{\R^3\times\R^3} f^\init |\widetilde{\Z}(t)-\Z(t)|^2.
\end{align}
Our results demand strong a priori assumptions (mostly on the fluid component) that are listed below. 

\begin{ass}\label{ass:1}
  For $k\in\{0,1,2\}$, $m_k \widetilde{f}$ and $m_k f$ belong to $\Ll^\infty(\R_+;\mathrm{L}^\infty(\R^3))$.
\end{ass}
  \begin{ass}\label{ass:2}
    $\widetilde{\Z}-\Z$ belongs to $\mathrm{L}^\infty(\R_+;\mathrm{L}^\infty(\R^3\times\R^3))$.
  \end{ass}
  \begin{ass}\label{ass:3}
    $\widetilde{u}$ belongs to $\Ll^2(\R_+;\mathrm{L}^\infty(\R^3))$.
  \end{ass}
  \begin{ass}\label{ass:4}
    $\widetilde{u}$ belongs to $\Ll^1(\R_+;\W^{1,\infty}(\R^3))$.
    \end{ass}
The first result evaluates the distance between two hypothetical solutions of the incompressible Navier-Stokes equation, supplied by a (perturbation of a) Brinkman-like force, assuming that the kinetic components, even though originating from the same initial data $f^\init$, are not \emph{a priori} linked to the fluid vector fields. In particular, the ODE satisfied by the trajectories $t\mapsto \Z(t)$ and $t\mapsto \widetilde{\Z}(t)$ are not involved in the following Lemma.
\begin{lem}\label{lem:dynw}
  We consider two pairs $(\widetilde{f},\widetilde{u})$ and $(f,u)$ such that $f=\textnormal{Z}\# f^\init$ and $\widetilde{f}=\widetilde{\textnormal{Z}}\#f^\init$ satisfying all four Assumptions~\ref{ass:1}, \ref{ass:2}, \ref{ass:3} and \ref{ass:4}. Fix a source term $G\in\Ll^2(\R_+;\dot{\H}^{-1}(\R^3))$, and assume that $\widetilde{u}$ and $u$ are incompressible Leray solutions of 
    \begin{align*}
      \partial_t \widetilde{u} + \widetilde{u}\cdot \nabla \widetilde{u} - \Delta \widetilde{u} + \nabla \widetilde{p}  &= \widetilde{j} - \widetilde{\rho} \,\widetilde{u} + G,\\
                                                                                                                              \partial_t u + u \cdot \nabla u - \Delta u + \nabla p  &= j - \rho u.
    \end{align*}
Then for $w:=\widetilde{u}-u$ and any $T>0$, there holds for $t\in [0,T]$
   \begin{multline*}
\|w(t)\|_2^2 + \int_0^t \|\nabla w(s)\|_2^2\,\dd s \lesssim \|w(0)\|_2^2 \\+ A(T)\int_0^t \big[1 +\|\nabla \widetilde{u}(s)\|_\infty\big]\|w(s)\|_2^2\,\dd s \\ + A(T)^2 \int_0^t \big[1+\|\widetilde{u}(s)\|_\infty^2+\|\nabla \widetilde{u}(s)\|_\infty\big]Q_{\widetilde{\Z},\Z}(s)\,\dd s + A(T)\int_0^t \|G(s)\|_{\dot{\H}^{-1}(\R^3)}^2\,\dd s, 
\end{multline*}
where the constant behind $\lesssim$ is universal and the constant $A(T)$ is explicitely
\begin{equation}\label{def:A}
 1+\sup_{s\in[0,T]} \Big\{\|\widetilde{\Z}(s)-\Z(s)\|_\infty^2 + \| m_2 \widetilde{f}(s)\|_\infty + \| m_2 f(s)\|_\infty + \| m_0 f(s)\|_\infty \Big\}.
\end{equation}
\end{lem}
The second building block relates the distance between the two kinetic components $f=\textnormal{Z}\# f^\init$ and $\widetilde{f}=\widetilde{\textnormal{Z}}\#f^\init$ to the distance between the two vector fields defining the Lagrangian flows $\textnormal{Z}$ and $\widetilde{\textnormal{Z}}$. This time, no solution of the Navier-Stokes equations is involved \emph{a priori}, while only the dynamics of the Lagrangian trajectories are used.
\begin{lem}\label{lem:dynQ}
Consider two pairs $(\widetilde{f},\widetilde{u})$ and $(f,u)$, such that Assumption~\ref{ass:1} holds together with Assumption~\ref{ass:2} and Assumption~\ref{ass:4}. Assume furthermore that $\widetilde{\Z}$ and $\Z$ are respectively the flows associated to the vector fields $(t,x,v) \mapsto (v,{u}(t,x)-v)$ and $(v,\widetilde{u}(t,x)-v)$. Then, for $w:=u_1-u_2$, there holds
  \begin{align*}
Q_{\widetilde{\Z},\Z}(t) \lesssim \int_0^t \Big[1+  \|\nabla \widetilde{u}(s)\|_\infty \Big] Q_{\widetilde{\Z},\Z}(s)\,\dd s  + \int_0^t \|m_0 f(s)\|_\infty \|w(s)\|_2^2\,\dd s.
  \end{align*}
\end{lem}
\begin{rem}
It is clear from the above statements that Assumption~\ref{ass:1} could be substantially relaxed.
\end{rem} 
\subsection{Proof of Lemma~\ref{lem:dynw}}
  \begin{proof}
For $q:=\widetilde{p}-p$ and $(\widetilde{F},F):=(\widetilde{j}-\widetilde{\rho}\,\widetilde{u},j-\rho u)$, there holds
    \begin{align*}
      \partial_t w + u\cdot \nabla w - \Delta w + \nabla q = - w \cdot \nabla \widetilde{u} + \widetilde{F}-F + G,
     \end{align*}
which, in the setting of Leray solutions, leads to the following distributional inequality 
     \begin{align*}
       \frac12 \frac{\dd}{\dd t}\|w(t)\|_2^2 + \|\nabla w(t)\|_2^2 
       \leq \int_{\R^3} |w \cdot w \cdot \nabla \widetilde{u}| +  \int_{\R^3} w \cdot (\widetilde{F}-F) + \int_{\R^3} w\cdot G\quad.
     \end{align*}
Thanks to the regularity assumption on $\widetilde{u}$, we infer 
     \begin{equation}\label{ineq:w}
       \frac12 \frac{\dd}{\dd t}\|w(t)\|_2^2 +\|\nabla w(t)\|_2^2 \leq \|w(t)\|^2 \|\nabla \widetilde{u}(t)\|_\infty + \stackrel{:=I}{\overbrace{\int_{\R^3} w \cdot (\widetilde{F}-F)}} + \int_{\R^3} w \cdot G\quad.
     \end{equation}
Let's omit for a moment the notation of the time variable $t$. The integral $I$ can be written in the following way using the definition of the Brinkman forces $\widetilde{F}$ and $F$
     \begin{align*}
I =  \iint_{\R^3\times\R^3} f^\init w(\widetilde{\X})\cdot (\widetilde{\V}-\widetilde{u}(\widetilde{\X})) - \iint_{\R^3\times\R^3} f^\init w(\X)\cdot (\V-u(\X)), 
     \end{align*}
where we introduced the decomposition $\widetilde{\Z} = (\widetilde{\X},\widetilde{\V})$ and $\Z = (\X,\V)$ together with a small abuse of notation, writing $\widetilde{u}(\widetilde{\X})$ and $u(\X)$ for respectively $\widetilde{u}(t,\widetilde{\X})$ and $u(t,\X)$. We have then
     \begin{multline*}
I = \iint_{\R^3\times\R^3} f^\init (w(\widetilde{\X})-w(\X))\cdot (\widetilde{\V}-\widetilde{u}(\widetilde{\X}))  - \iint_{\R^3\times\R^3} f^\init w(\X)\cdot (\widetilde{u}(\widetilde{\X})-\widetilde{\V}+\V-u(\X)),
\end{multline*}
and since $w=\widetilde{u}-u$
\begin{multline}\label{ineq:I}
I \leq  \iint_{\R^3\times\R^3} f^\init (w(\widetilde{\X})-w(\X))\cdot (\widetilde{\V}-\widetilde{u}(\widetilde{\X}))  \\ + \iint_{\R^3\times\R^3} f^\init w(\X) \cdot (\widetilde{\V}-\V) \\ + \iint_{\R^3\times\R^3} f^\init w(\X)\cdot (\widetilde{u}(\X)-\widetilde{u}(\widetilde{\X})).
\end{multline}
Let's recall a useful estimate involving the maximal function, that was already crucial in our previous work \cite{hkm3}. For $x,y\in\R^3$ and any function $\ffi\in\H^1(\R^3)$, there holds
\[|\ffi(x)-\ffi(y)|\lesssim |x-y| \big[M|\nabla \ffi|(x) + M|\nabla \ffi|(y)\big],\]
for a constant behind $\lesssim$ which depends only on the dimension. In particular, we get from \eqref{ineq:I}
\begin{multline*}
I \lesssim  \iint_{\R^3\times\R^3} f^\init |\widetilde{\X}-\X| |\widetilde{\V}| \big[M|\nabla w|(\widetilde{\X}) + M|\nabla w|(\X)\big] \\ + \iint_{\R^3\times\R^3} f^\init|\widetilde{\X}-\X||\widetilde{u}(\widetilde{\X})| \big[M|\nabla w|(\widetilde{\X}) + M|\nabla w|(\X)\big] \\ + \iint_{\R^3\times\R^3} f^\init w(\X) \cdot (\widetilde{\V}-\V)  + \iint_{\R^3\times\R^3} f^\init w(\X)\cdot (\widetilde{u}(\X)-\widetilde{u}(\widetilde{\X})).
\end{multline*}
Using that $\widetilde{\Z}-\Z$ belongs to $\Ll^\infty(\R_+;\mathrm{L}^\infty(\R^3\times\R^3))$ and $\widetilde{u}\in\Ll^1(\R_+;\W^{1,\infty}(\R^3))$ by assumption, we have
\begin{multline*}
I \lesssim  \iint_{\R^3\times\R^3} f^\init |\widetilde{\X}-\X| |\widetilde{\V}| M|\nabla w|(\widetilde{\X}) + \iint_{\R^3\times\R^3} f^\init |\widetilde{\X}-\X| (\|\widetilde{\Z}-\Z\|_\infty+|\V|)M|\nabla w|(\X)\big] \\ + \iint_{\R^3\times\R^3} f^\init|\widetilde{\X}-\X| |\widetilde{u}(\widetilde{\X})| \big[M|\nabla w|(\widetilde{\X}) + M|\nabla w|(\X)\big] \\ + \iint_{\R^3\times\R^3} f^\init |w(\X)| |\widetilde{\V}-\V|  + \iint_{\R^3\times\R^3} f^\init  \|\nabla \widetilde{u}\|_\infty |w(\X)| |\X-\widetilde{\X}|,
\end{multline*}
where except $f^\init$, every function involved in the right hand side depends on the time variable. For any $\ep \in (0,1)$, we infer from Young's inequality, writing simply $Q$ for $Q_{\widetilde{\Z},\Z}$, we get using Young's inequality 
\begin{multline*}
  I \lesssim  \ep^{-1}\big[1+\|\widetilde{\Z}-\Z\|_\infty^2 + \|\widetilde{u}\|_\infty^2\big] Q   + \big[1+\|\nabla\widetilde{u}\|_\infty\big]Q \\+ \ep \big[\|m_2 \widetilde{f}\|_\infty + \|m_2 f\|_\infty +1\big] \| M\nabla w\|_2^2  \\  + \|m_0 f\|_\infty\big[1 + \|\nabla \widetilde{u}\|_\infty\big] \| w\|_2^2.
\end{multline*}
Adding back the time variable $t$ that we omitted before and recalling the definition \eqref{def:A} of $A(T)$, the stability of the maximal function on $\mathrm{L}^2(\R^3)$ allows us to write for yet another universal symbol $\lesssim$, since $\ep\in(0,1)$
\begin{equation*}
  I(t) \lesssim  \ep^{-1}\big[A(T) + \|\widetilde{u}\|_\infty^2 + \|\nabla \widetilde{u}\|_\infty\big] Q + \ep A(T) \| \nabla w\|_2^2  + A(T)\big[1 + \|\nabla \widetilde{u}\|_\infty\big] \| w\|_2^2.
\end{equation*}
Going back to \eqref{ineq:w} and using once more Young's inequality for the same arbitrary $\ep\in (0,1)$ to handle the term involving $G$ we have 
\begin{multline*}
      \frac12 \frac{\dd}{\dd t}\|w(t)\|_2^2 +\|\nabla w(t)\|_2^2 
      \leq I(t) + \|w(t)\|^2 \|\nabla \widetilde{u}(t)\|_\infty + \frac{\ep}{2} \|\nabla w(t)\|_2^2 +  \frac{\ep^{-1}}{2}\|G(t)\|_{\dot{\H}^{-1}(\R^3)}^2.
\end{multline*}
All in all, since $A(T)\geq 1$,  we have
\begin{multline*}
      \frac12 \frac{\dd}{\dd t}\|w(t)\|_2^2 +\|\nabla w(t)\|_2^2 \lesssim  \ep^{-1}\big[A(T) + \|\widetilde{u}\|_\infty^2 + \|\nabla \widetilde{u}\|_\infty\big] Q\\ \hspace{4cm} + \ep A(T) \| \nabla w\|_2^2 + \ep^{-1} \|G(t)\|_{\dot{\H}^{-1}(\R^3)}^2 \\ + A(T)\big[1 + \|\nabla \widetilde{u}\|_\infty\big] \| w\|_2^2.
 \end{multline*}
 Choosing $\ep$ small enough we recover, for another universal symbol $\lesssim$, we get
 \begin{multline*}
      \frac12 \frac{\dd}{\dd t}\|w(t)\|_2^2 +\|\nabla w(t)\|_2^2 \lesssim  A(T)^2 \big[1 + \|\widetilde{u}\|_\infty^2 + \|\nabla \widetilde{u}\|_\infty\big] Q\\  +A(T) \|G(t)\|_{\dot{\H}^{-1}(\R^3)}^2  + A(T)\big[1 + \|\nabla \widetilde{u}\|_\infty\big] \| w\|_2^2,
 \end{multline*}
which is exaclty the estimate written in differential form. $\qedhere$
\end{proof}
\subsection{Proof of Lemma~\ref{lem:dynQ}}
\begin{proof}
  We have, writing simply $Q$ for $Q_{\widetilde{\Z},\Z}$, by a direct computation
  \begin{multline*}
Q'(t) = \iint_{\R^3\times\R^3} f^\init (\widetilde{\X}^t-\X^t)\cdot (\widetilde{\V}^t-\V^t) \\+ \iint_{\R^3\times\R^3} f^\init (\widetilde{\V}^t-\V^t)\cdot (\widetilde{u}(\widetilde{\X}^t)-u(\X^t)) - \iint_{\R^3\times\R^3} f^\init |\widetilde{\V}^t-\V^t|^2.
             \end{multline*}
Thanks to the Young inequality we deduce that
             \begin{equation*}
Q'(t) \lesssim Q(t) + \iint_{\R^3\times\R^3} f^\init |\widetilde{\V}^t -\V^t|\,|\widetilde{u}(\widetilde{\X}^t)-\widetilde{u}(\X^t)| + \iint_{\R^3\times\R^3} f^\init |\widetilde{u}(\X^t)-u(\X^t)|^2, 
\end{equation*}
and thus
\begin{align*}
Q'(t) \lesssim  \Big[1+\|\nabla \widetilde{u}(t)\|_\infty\Big] Q(t)  + \int_{\R^3} m_0 f (t)|(\widetilde{u}-u)(t)|^2,
     \end{align*}
     hence the conclusion.
   \end{proof}

           \section{Proof of Theorem~\ref{thm:univns}}
\label{sec:proof1}
             We consider two Leray solutions $(f_1,u_1)$ et $(f_2,u_2)$ of the VNS sharing the same initial data $(f^\init,u^\init)$, with $(x,v)\mapsto |v|^6 f^\init(x,v)$ integrable. Thanks to Lemma~\ref{lem:reg}, we have $m_\ell f_k \in\Ll^\infty(\R_+;\mathrm{L}^\infty(\R^3))$ for $(\ell,k)\in\{0,1,2\}\times\{1,2\}$. We use the  notation $F_k := j_k - u_k \rho_k$ for the two associated Brinkman forces and similarly, we write $\Z_k:=(\X_k,\V_k)$ for the characteristic curves associated with the vector field $(t,x)\mapsto (v,u_k(t,x)-v)$. Just as for the proof of \cite[Lemma 4.13]{hkm3}, we note that
             \[|\V_1(t,x,v) - \V_2(t,x,v)|\leq \int_0^t |u_1(s,\X_1(s))|\,\dd s + \int_0^t |u_2(s,\X_2(s))|\,\dd s,\]
             from which we infer (by direct integration) $\Z_1-\Z_2\in\Ll^\infty(\R_+;\mathrm{L}^\infty(\R^3\times\R^3))$, using point $(i)$ of Lemma~\ref{lem:reg}. This justifies therefore \textbf{Assumption \ref{ass:2}} if we aim to use Lemma~\ref{lem:dynw} to evaluate the distance between two fluid/kinetic pairs involving $f_1$ and $f_2$ respectively and this is precisely what we're about to do.

             \medskip
             
             Without loss of generality, we assume that $u_1$ is well-approximated by a smooth dyadic approximate identity $(\psi_j)_j$. A direct computation shows that $u_{1\star j}:=u_1\star \psi_j$ solves
\begin{align}\label{eq:u1j}
 \partial_t u_{1\star j}  + u_{1\star j} \cdot \nabla u_{1\star j} - \Delta u_{1\star j} +\nabla p_1\star \psi_j = j_1 -\rho_1 u_{1\star j} + G_j, 
\end{align}
where the error term $G_j$ is given by 
\begin{align}\label{def:Gj}
  G_j:= F_1 \star\psi_j - F_1  - \rho_1 (u_1 -u_1 \star \psi_j) + u_1 \star \psi_j \cdot \nabla (u_1 \star \psi_j) - (u_1 \cdot \nabla u_1)\star \psi_j.\end{align}
Since $\rho_k$ and $j_k$ all belong to $\Ll^\infty(\R_+;\mathrm{L}^\infty(\R^3))$ for $k=1,2$, and $u_1\star\psi_j$ belongs to $\Ll^2(\R_+;\W^{1,\infty}(\R^3))$, we can invoke Lemma~\ref{lem:dynw} for the fluid/kinetic pairs $(f_1,u_{1\star j})$ and $(f_1,u_1)$ (of course $u_1\star\psi_j$ plays the role of the smooth function). Note here that in this $\widetilde{\Z} = \Z = \Z_1$ (and thus $Q_{\widetilde{\Z},\Z}=0$) so for $T>0$ and $t\in[0,T]$ we have 
\begin{multline}
\label{ineq:w1j}\|(u_1-u_{1\star j})(t) \|_2^2 + \int_0^t \|\nabla (u_1-u_{1\star j})(s)\|_2^2\,\dd s \lesssim \|(u_1-u_{1\star j})(0)\|_2^2 \\+ A(T) \int_0^t \big[1 +\|\nabla u_{1\star j}(s)\|_\infty\big]\|(u_1-u_{1\star j})(s)\|_2^2\,\dd s  + \int_0^t \|G_j(s)\|_{\dot{\H}^{-1}(\R^3)}^2\,\dd s,
\end{multline}
where
\begin{align*}
A(T) := 1 + \sup_{s\in[0,T]} \Big\{ \|m_2 f_1(s)\|_\infty + \|m_0 f_1(s)\|_\infty \Big\}.
\end{align*}
Now, let's proceed in the same fashion to measure the gap $u_1\star \psi_j-u_2 = u_{1\star j}-u_2$, using Lemma~\ref{lem:dynw} for the pairs $(f_1,u_{1\star j})$ and $(f_2,u_2)$. Of course, once again,  $u_{1\star j}$ plays the role of the smooth function. But this time $Q_{\widetilde{\Z},\Z} = Q_{\Z_1,\Z_2}$ has a non-trivial contribution in the estimate (note that we verified above that \textbf{Assumption \ref{ass:2}} is valid). More precisely, we have
   \begin{multline}\label{ineq:u1ju2}
\|(u_{1\star j}-u_2)(t)\|_2^2 + \int_0^t \|\nabla (u_{1\star j}-u_2)(s)\|_2^2\,\dd s \lesssim \|(u_{1\star j}-u_2)(0)\|_2^2 \\+ A(T)\int_0^t \big[1 +\|\nabla u_{1\star j}(s)\|_\infty\big]\|(u_{1\star j}-u_2)(s)\|_2^2\,\dd s \\ + A(T)^2 \int_0^t \big[1+\|u_{1\star j}(s)\|_\infty^2+\|\nabla u_{1\star j}(s)\|_\infty\big]Q_{\Z_1,\Z_2}(s)\,\dd s\\ + A(T)\int_0^t \|G_j(s)\|_{\dot{\H}^{-1}(\R^3)}^2\,\dd s, 
\end{multline}
where the constant behind $\lesssim$ is universal and the constant $A(T)$ is updated with the following value
\begin{multline}
 \label{eq:vraiA} 1+\sup_{s\in[0,T]} \Big\{\|\Z_1(s)-\Z_2(s)\|_\infty^2 + \| m_2 f_1(s)\|_\infty + \| m_2 f_2(s)\|_\infty + \| m_0 f_1(s)\|_\infty +\| m_0 f_2(s)\|_\infty \Big\}.
\end{multline}
For later use, let us remark that by Lemma~\ref{lem:reg}, we have the bound
\begin{equation}
\label{eq:boundA}
    A(T) \leq \Psi_0(T, (E^\init_k)_{k=1,2}, (M_6 f^\init_k)_{k=1,2}, (N_q (f^\init_k))_{k=1,2}),
\end{equation}
where $\Psi_0$ is a continuous function.

Now, to close the argument we need to include the dynamic of $Q_{\Z_1,\Z_2}$ in the estimate. However, we cannot directly use Lemma~\ref{lem:dynQ} because neither $u_1$ or $u_2$ has sufficient regularity to ensure \textbf{Assumption \ref{ass:4}}. For this reason, we introduce the lagrangian trajectory $\Z_{1\star j}$ associated to the regular vector field $(t,x,v)\mapsto (v,u_{1\star j}(t,x)-v)$ and first estimate $Q_{\Z_1,\Z_{1\star j}}$ and $Q_{\Z_{1\star j},\Z_2}$. Lemma~\ref{lem:dynQ} can indeed be used to control these two functionals (using each time $\Z_{1\star j}$ for the smooth function $\widetilde{\Z}$) and thus, for some universal symbol $\lesssim$ and $t\in[0,T]$
  \begin{align*}
    Q_{\Z_1,\Z_{1\star j}}(t) &\lesssim \int_0^t \Big[1+  \|\nabla u_{1\star j}(s)\|_\infty \Big] Q_{\Z_1,\Z_{1\star j}}(s)\,\dd s  + A(T)\int_0^t  \|(u_1-u_{1\star j})(s)\|_2^2\,\dd s\\
    Q_{\Z_{1\star j},\Z_2}(t) &\lesssim \int_0^t \Big[1+  \|\nabla u_{1\star j}(s)\|_\infty \Big] Q_{\Z_{1\star j},\Z_2}(s)\,\dd s  + A(T)\int_0^t \|(u_{1\star j}-u_2)(s)\|_2^2\,\dd s.
  \end{align*}
  We're eventually reaching the end of our computation. We note 
  that $Q_{\Z_1,\Z_2} \leq 2 \widetilde{Q}_j := Q_{\Z_1,\Z_{1\star j}} + Q_{\Z_{1\star j},\Z_2}$. Then, we define \begin{align}\label{eq:defmu}\mu_j(t) &:= \|(u_1-u_{1\star j})(t)\|_2^2 +\|(u_{1\star j}-u_2)(t)\|_2^2,\\ \label{eq:defnu} \nu_j(t)&:=\|\nabla (u_1-u_{1\star j})(t)\|_2^2 +\|\nabla (u_{1\star j}-u_2)(t)\|_2^2.\end{align} Lastly, we simply add up the two estimates above on $Q_{\Z_1,\Z_{1\star j}}$ and $Q_{\Z_{1\star j},\Z_2}$ together with \eqref{ineq:w1j} and \eqref{ineq:u1ju2} and recover
  \begin{multline*}
\widetilde{Q}_j(t) + \mu_j(t) + \int_0^t \nu_j(s)\,\dd s \lesssim \mu(0)  + A(T)\int_0^t \|G_j(s)\|_{\dot{\H}^{-1}(\R^3)}^2\,\dd s \\\hspace{2cm}+ A(T) \int_0^t \big[ 1+ \|\nabla u_{1\star j}(s)\|_\infty\big] \mu_j(s)\,\dd s \\ + A(T)^2 \int_0^t \big[1+\|u_{1\star j}(s)\|_\infty^2 +\|\nabla u_{1\star j}(s)\|_\infty\big] \widetilde{Q}_j(s)\,\dd s.
  \end{multline*}
  Using $A(T)\geq 1$, this estimates entails the following Gr\"onwall-friendly estimate for $\xi_j(t) := \widetilde{Q}_j(t) + \mu_j(t)$
  \begin{multline*}\xi_j(t) + \int_0^t \nu_j(s)\,\dd s \lesssim \xi_j(0) + A(T)^2 \int_0^t \|G_j(s)\|_{\dot{\H}^{-1}(\R^3)}^2\,\dd s \\+ A(T)^2 \int_0^t \big[1+\|u_{1\star j}(s)\|_\infty^2 + \|\nabla u_{1\star j}(s)\|_\infty\big] \xi_j(s)\,\dd s,\end{multline*}
  which thus imply, for some universal constant $\C$
  \begin{align*}
\xi_j(t) + \int_0^t \nu_j(s)\,\dd s \leq \exp(\C A(T)^2 B_j(T)) \left(\xi_j(0)+\int_0^t \|G_j(s)\|_{\dot{\H}^{-1}(\R^3)}^2\,\dd s \right). 
  \end{align*}
where
\[B_{j}(T) := 1+T+\int_0^T \left(\|u_{1\star j}(s)\|_\infty^2+ \|\nabla u_{1\star j}(s)\|_\infty\right)\,\dd s.\]
Recalling the definition of $\mu_j$ and $\nu_j$ in \eqref{eq:defmu} -- \eqref{eq:defnu} and using $u_1(0) = u_2(0) = u^\init$ we eventually have by triangular inequality for $w=u_1-u_2$
  \begin{multline}\label{ineq:qz1z2}
Q_{\Z_1,\Z_2}(t) + \|w(t)\|_2^2+ \int_0^t \|\nabla w(s)\|_2^2 \,\dd s \\\leq \exp(\C A(T)^2 B_j(T)) \left[\|u^\init-u^\init\star \psi_j\|_2^2 +\int_0^t \|G_j(s)\|_{\dot{\H}^{-1}(\R^3)}^2\,\dd s \right],
  \end{multline}
  where $B_j(T)$ is defined just as above and $A(T)$ by \eqref{eq:vraiA}.

  \begin{rem}\label{rem:gronstab}
A close look of the previous proof leading to \eqref{ineq:qz1z2} shows that we control actually more than merely $Q_{\Z_1,\Z_2}$ : the Grönwall estimate is obtained on $\xi_j = \widetilde{Q}_j+\mu_j$ so we actually also control  $Q_{\Z_1,\Z_{1\star j}}$ and $Q_{\Z_{1\star j},\Z_2}$. This remark will be useful in the proof of Theorem~\ref{thm:stabvns}.
  \end{rem}

  Now, the whole point is to understand how the competition ends between the exponential term (which diverges) and the bracket one (which goes to $0$), as $j\rightarrow +\infty$. Let's deal first with the second one and prove that it converges exponentially to $0$ that is, for some $\beta>0$, it behaves asymptotically as $\textnormal{O}(e^{-\beta j})$. For the term $\|u^\init - u^\init \star \psi_j\|_2^2$, this follows directly from the fact that $u^\init\in\H^s(\R^3)$ for some $s>0$: since $(\psi_{j})_j$ is a dyadic extraction of an approximate identity, a standard result on convolution ensures for some $s>0$ that $\|u^\init - u^\init \star \psi_j \|_2^2 \lesssim 4^{-s j} = e^{-s\ln(4)j}$. For $G_j$, we recall its definition \eqref{def:Gj} that we reproduce here
\[G_j = \stackrel{G_j^1}{\overbrace{F_1 \star \psi_j-F_1}} - \stackrel{G_j^2}{\overbrace{\rho_1(u_1-u_{1\star j})}} + \stackrel{G_j^3}{\overbrace{u_{1\star j}\cdot \nabla u_{1\star j} - (u_1 \cdot\nabla u_1) \star \psi_j}}.\] 
For $G_j^1$, since $f^\init$ has finite moments up to the sixth order, we can do just as in \cite[Lemma 4.7]{ehkm} to check that $F_1$ belongs to $\Ll^2(\R_+;\mathrm{L}^2(\R^3))$. Then, as before (but with shifted regularity), we have $\|F_1(s)\star \psi_j -F_1(s)\|_{\dot{\H}^{-1}(\R^3)} \leq 2^{-j} \|F_1(s)\|_2$ and the exponential decay follows. Almost the same story happens for $G_j^2$ using the Sobolev embedding $\dot{\H}^1(\R^3)\hookrightarrow\mathrm{L}^6(\R^3)$
$\|G_j^2(s)\|_{\dot{\H}^{-1}(\R^3)} \leq \|G_j^2(s)\|_{6/5} \leq \|\rho_1(s)\|_3 \|u_1(s)-u_1(s)\star \psi_j\|_{2} \leq A(T) 2^{-j}\|\nabla u_1(s)\|_2$ which allows to conclude because $u_1$ belongs to $\Ll^2(\R_+;\H^1(\R^3))$, as any Leray solution. Lastly, due to the incompressibilty condition, $G_j^3$ equals  $\div(u_{1\star j} \otimes u_{1\star j} - (u_1\otimes u_1)\star \psi_j)$ where $u_{1\star j} = u_1 \star \psi_j$ so that
\[\int_0^T \|G_j^3(s)\|_{\dot{\H}^{-1}(\R^3)}^2\,\dd s \leq \int_0^T \|(u_1\star \psi_j)\otimes (u_1\star \psi_j)-(u_1\otimes u_1)\star \psi_j\|_2^2 \,\dd s.\]
Since $u_1$ is assumed to be well-approximated by $(\psi_j)_j$ in the sense of Definition~\ref{def:well}, we get the expected exponential decay.

\medskip

We have therefore established, up to an universal constant encoded in $\lesssim$, the following estimate for some $\beta >0$ 
  \begin{align*}
Q_{\Z_1,\Z_2}(t) + \|w(t)\|_2^2+ \int_0^t \|\nabla w(s)\|_2^2 \,\dd s \lesssim \exp(\C A(T)^2 B_j(T)) \exp(-\beta j).
  \end{align*}
But, using once more that $u_1$ is well-approximated, we have the existence of a positive sequence $(a_k)_k \in c_0(\N)$ such that
  \[B_j(T) \leq 1+ T +\sum_{k=0}^j a_k.\]
  For some integer $N$  we have $k> N\Rightarrow C A(T)^2 a_k<\beta/2$ which proves that, for $j> N$,
  \begin{align*}
    \exp(\C A(T)^2B_j(T)) \exp &(-\beta j) \\
    &= \exp\left(\C A(T)^2\sum_{k=0}^{N} a_k\right) \exp\left(\C A(T)^2\sum_{k=N+1}^j a_k\right) \exp(-\beta j) \\&\leq  \exp\left(\C A(T)^2\sum_{k=0}^{N} a_k\right) \exp(\beta(j-N)/2) \exp(-\beta j)\\
    &= \exp\left(\C A(T)^2\sum_{k=0}^{N} a_k\right) \exp(-\beta N/2) \exp(-\beta j/2) \conv{j}{+\infty} 0,
  \end{align*}
  which concludes the proof of Theorem~\ref{thm:univns}. $\qedhere$

  \section{Proof of Theorem~\ref{thm:stabvns}}
  \label{sec:proof4}
  
  The strategy is perfectly similar to the one followed for the proof of Theorem~\ref{thm:univns}. The estimate for $u_1-u_{1\star j}$ is rigorously the same, leading to \eqref{ineq:w1j}, where $G_j$ is given by \eqref{def:Gj}. For $u_{1\star j}-u_2$, there is an extra term coming from the fact that $f^\init_1\neq f^\init_2$. In order to invoke directly Lemma~\ref{lem:dynw}, we rewrite \eqref{eq:u1j} as 
  \begin{align*}
 \partial_t u_{1\star j}  + u_{1\star j} \cdot \nabla u_{1\star j} - \Delta u_{1\star j} +\nabla p_1\star \psi_j = j_{1\#2} -\rho_{1\#2} u_{1\star j} + G_j + H_j, 
\end{align*}
where $j_{1\#2}$ and $\rho_{1\#2}$ are the moments associated to $f_{1\#2}$, the push-forward of the initial data $f^\init_2$ through the $\Z_1$ trajectory, that is $f_{1\#2}(t) = \Z_1(t) \# f^\init_2$. Of course the extra error term $H_j$ is simply given by
\begin{align}\label{eq:hj}
  H_j = j_1-j_{1\#2} - (\rho_{1}-\rho_{1\#2})u_{1\star j}.
\end{align}
Now, we can directly use Lemma~\ref{lem:dynw} to evaluate the distance $u_{1\star j}-u_2$, taking $G_j+H_j$ as error term. Note that we invoke this lemma with the common initial data $f^\init_2$ for the kinetic components, so that the functional $Q_{\widetilde{\Z},\Z}(t)$ is here 
\begin{align*}
Q_{2;\Z_{1},\Z_2}(t) := \iint_{\R^3\times\R^3}  |\Z_{1}(t,x,v)-\Z_2(t,x,v)|^2f^\init_2(x,v)\,\dd x\, \dd v.
\end{align*}
We have therefore 
   \begin{multline*}
\|(u_{1\star j}-u_2)(t)\|_2^2 + \int_0^t \|\nabla (u_{1\star j}-u_2)(s)\|_2^2\,\dd s \lesssim \|(u_{1\star j}-u_2)(0)\|_2^2 \\+ A(T)\int_0^t \big[1 +\|\nabla u_{1\star j}(s)\|_\infty\big]\|(u_{1\star j}-u_2)(s)\|_2^2\,\dd s \\ + A(T)^2 \int_0^t \big[1+\|u_{1\star j}(s)\|_\infty^2+\|\nabla u_{1\star j}(s)\|_\infty\big]Q_{2;\Z_{1},\Z_2}(s)\,\dd s\\ + A(T)\int_0^t \|(G_j+H_j)(s)\|_{\dot{\H}^{-1}(\R^3)}^2\,\dd s. 
\end{multline*}
For the dynamics of $Q_{2;\Z_{1},\Z_2}$, we proceed as we have done in the proof of Theorem~\ref{thm:univns} interleaving the smooth trajectory $\Z_{1\star j}$ in between $\Z_1$ and $\Z_2$, and recovering instead \eqref{ineq:qz1z2}, the following estimate 
  \begin{multline}\label{ineq:qz1z2w}
Q_{2;\Z_{1},\Z_2}(t) + \|w(t)\|_2^2+ \int_0^t \|\nabla w(s)\|_2^2 \,\dd s \\\leq \exp(\C A(T)^2 B_j(T)) \left[\|u^\init_1-u^\init_1\star \psi_j\|_2^2 +\int_0^t \|G_j(s)\|_{\dot{\H}^{-1}(\R^3)}^2\,\dd s \right] \\
    + \exp(\C A(T)^2 B_j(T)) \left[\|u^\init_1-u^\init_2\|_2^2 +\int_0^t \|H_j(s)\|_{\dot{\H}^{-1}(\R^3)}^2\,\dd s \right].
  \end{multline}
  Now, as noticed in Remark~\ref{rem:gronstab}, a precise inspection of the proof gives actually more than this: in the previous estimate we can replace $Q_{2;\Z_1,\Z_2}$ by $Q_{2;\Z_{1},\Z_{1\star j}} +Q_{2;\Z_{1\star j},\Z_2}$ with the following definition of these functionals 
\begin{align}
\label{eq:Q2z1zj}  Q_{2;\Z_1,\Z_{1\star j}}(t) &:= \iint_{\R^3\times\R^3}  |\Z_1(t,x,v)-\Z_{1\star j}(t,x,v)|^2f^\init_2(x,v)\,\dd x\, \dd v,\\
  \label{eq:Q2z2zj}  Q_{2;\Z_{1\star j},\Z_2}(t) &:= \iint_{\R^3\times\R^3}  |\Z_{1\star j}(t,x,v)-\Z_2(t,x,v)|^2f^\init_2(x,v)\,\dd x\, \dd v,
\end{align}
where as before $\Z_{1\star j}$ denotes the lagrangian trajectory associated to the regular vector field $(t,x)\mapsto (v,u_{1\star j}(t,x)-v)$. Introducing the notation $\widetilde{Q}_j:=Q_{2;\Z_{1},\Z_{1\star j}} +Q_{2;\Z_{1\star j},\Z_2}$ we thus also have 
  \begin{multline}\label{ineq:qz1z2wow}
\widetilde{Q}_j(t) + \|w(t)\|_2^2+ \int_0^t \|\nabla w(s)\|_2^2 \,\dd s \\\leq \exp(\C A(T)^2 B_j(T)) \left[\|u^\init_1-u^\init_1\star \psi_j\|_2^2 +\int_0^t \|G_j(s)\|_{\dot{\H}^{-1}(\R^3)}^2\,\dd s \right] \\
    + \exp(\C A(T)^2 B_j(T)) \left[\|u^\init_1-u^\init_2\|_2^2 +\int_0^t \|H_j(s)\|_{\dot{\H}^{-1}(\R^3)}^2\,\dd s \right].
  \end{multline}
Let's now bound the right-hand side of \eqref{ineq:qz1z2wow}. For the $H_j$ term, owing to the Sobolev embedding ${\dot{\H}^{1}(\R^3)\hookrightarrow\mathrm{L}^{6}(\R^3)}$, we only need to control the $\mathrm{L}^{6/5}(\R^3)$ norm of $H_j$. Turning back to the definition of \eqref{eq:hj} and using $\frac56 = \frac12+\frac13$ we have 
  \begin{multline*}
 \int_0^t \|H_j(s)\|^2_{\mathrm{L}^{6/5}(\R^3)}\,\dd s 
    \leq \int_0^t \|j_1(s)-j_{1\#2}(s)\|_{\mathrm{L}^{6/5}(\R^3)}^2\,\dd s \\+\sup_{s\in[0,t]} \|u_1(s)\|_{\mathrm{L}^2(\R^3)}^2 \int_0^t \|\rho_1(s)-\rho_{1\#2}(s)\|_{\mathrm{L}^3(\R^3)}^2 \,\dd s.
  \end{multline*}
  Since $M_1 f_1^\init$, $M_1 f_2^\init$ and $N_q(f_1^\init-f_2^\init)$ are finite for some $q>4$, we recover from the Definition~\ref{def:NR} of $\mathcal{N}_{T,R}$
  \begin{align*}
    \int_0^t \|H_j(s)\|^2_{\mathrm{L}^{6/5}(\R^3)}\,\dd s &\leq T \mathcal{N}_{T,R}(f_1^\init-f_2^\init)^2  + T \sup_{s\in[0,t]} \|u_1(s)\|_{\mathrm{L}^2(\R^3)}^2 \mathcal{N}_{T,R}(f_1^\init-f_2^\init)^2\\
    &\lesssim \mathcal{N}_{T,R}(f_1^\init-f_2^\init)^2, 
\end{align*}
where the constant in $\lesssim$ depends only on $T$ and the $\mathrm{L}^\infty(0,T;\mathrm{L}^2(\R^3))$ norm of $u_1$, while its $\mathrm{L}^1(0,T;\mathrm{L}^\infty(\R^3)$ norm defines the value of $R$ in the norm $\mathcal{N}_{T,R}$. Now, up to a symbol $\lesssim$ which depends only on the initial data and $u_1$, we infer from \eqref{ineq:qz1z2wow} 
\begin{multline}\label{ineq:w:NR}
 \|w(t)\|_2^2+ \int_0^t \|\nabla w(s)\|_2^2 \,\dd s + \widetilde{Q}_j(t) \\\lesssim \exp(\C A(T)^2 B_j(T)) \left[\|u^\init_1-u^\init_1\star \psi_j\|_2^2 +\int_0^t \|G_j(s)\|_{\dot{\H}^{-1}(\R^3)}^2\,\dd s \right] \\
    + \exp(\C A(T)^2 B_j(T)) \Big[\|u^\init_1-u^\init_2\|_2^2 + \mathcal{N}_{T,R}(f_1^\init-f_2^\init)^2 \Big].
  \end{multline}
  Now, as we noticed in the proof of Theorem~\ref{thm:univns}, the (large) bracket in the second line of \eqref{ineq:w:NR} enjoys an exponential decay $\exp(-\beta j)$ for some $\beta>0$ which depends only on the Sobolev regularity of $u_1$ and the parameter $\alpha$ given by its well-approximation property. Also, because of this property, we have the existence of a positive sequence $(a_k)_k\in \textnormal{c}_0(\N)$ such that 
  \[B_j(T) \leq 1 + T + \sum_{k=0}^j a_k.\]
  All in all, we can summarize  the previous estimate \eqref{ineq:w:NR} as
  \begin{align}\label{ineq:wdeltaD}
 \|w(t)\|_2^2+ \int_0^t \|\nabla w(s)\|_2^2 \,\dd s+ \widetilde{Q}_j(t) \leq \delta_j + h_j \Delta^\init,
   \end{align}
  where
  \begin{align*}
\Delta^\init = \|u^\init_1-u^\init_2\|_2^2 + \mathcal{N}_{T,R}(f_1^\init-f_2^\init)^2,
  \end{align*}
  while $(\delta_j)_j$ and $(h_j)_j$ are two sequences of positive numbers (depending only on $u_1$ and the initial data) the first one decreasing to $0$, the second one increasing to $+\infty$ but dominated by the first one in the sense that $(\delta_j h_j^\gamma)_j \rightarrow 0$ for any $\gamma>0$. Note that $h_j$ can actually be chosen exponential: $h_{j+1}/h_j$ is thus bounded so that  (and this is crucial) we have in particular $(\delta_j h_{j+1}^\gamma)_j  \rightarrow 0$ for all $\gamma>0$.
  
  \medskip
  
  Now, fix $\ep>0$ as in the statement of Theorem~\ref{thm:stabvns}. We have thus in particular $\delta_j h_{j+1}^{\ep^{-1}(1-\ep)} \leq 1$ for $j\geq J$ large enough, and this threshold $J$ depends only on $u_1$ and the initial data. This threshold being fixed, we can now pick initial data $(f^\init_2,u_2^\init)$ close enough to $(f^\init_1,u_1^\init)$ so that $h_J \leq (\Delta^{\init})^{-\ep}$. This smallness constraint on $\Delta^\init$ is translated thanks to a function $\Psi$ as in the statement of Theorem~\ref{thm:stabvns}. The precise claimed dependence of $\Psi$ is obtained thanks to~\eqref{eq:boundA}.
  
  Now consider $\ell:=\max\{j\geq J\,: h_j\leq (\Delta^\init)^{-\ep}\}$. Of course, $h_{\ell+1}>(\Delta^\init)^{-\ep}$. Owing to this choice of $\ell$ we have 
    \begin{align*}
\delta_\ell + h_\ell \Delta^\init  &\leq (\delta_\ell h_{\ell+1}^{\ep^{-1}(1-\ep)}) h_{\ell+1}^{-\ep^{-1}(1-\ep)} + h_{\ell} \Delta^\init \\
&\leq h_{\ell+1}^{-\ep^{-1}(1-\ep)} + (\Delta^\init)^{1-\ep}\\
&\leq 2(\Delta^{\init})^{1-\ep},
  \end{align*}
  because $h_{\ell+1}> (\Delta^\init)^{-\ep}$. Using \eqref{ineq:wdeltaD}, we thus obtained for some large and fixed integer $j$ and some constant $\textnormal{C}$
    \begin{align}\label{ineq:wdeltaD:suite}
 \|w(t)\|_2^2+ \int_0^t \|\nabla w(s)\|_2^2 \,\dd s+\widetilde{Q}_j(t) \leq \textnormal{C} \Big[\|u^\init_1-u^\init_2\|_2^2 + \mathcal{N}_{T,R}(f_1^\init-f_2^\init)^2\Big]^{1-\ep}.
    \end{align}
    This already sets the proof of the first stability estimate stated in Theorem~\ref{thm:stabvns}. In the case when $f_1^\init$ and $f_2^\init$ share the same mass, we will use Proposition~\ref{prop:norm-W}. First, recall the definitions $\widetilde{Q}_j := Q_{2;\Z_{1},\Z_{1\star j}} +Q_{2;\Z_{1\star j},\Z_2}$ and \eqref{eq:Q2z1zj} -- \eqref{eq:Q2z2zj}. Since $\Z_{1\star j}$ is associated to the smooth (and now fixed) vector field $(t,x,v)\mapsto(v,u_{1\star j}(t,x)-v)$, we can invoke Proposition~\ref{prop:norm-W} to control squared Wasserstein distances in term of these functionals. More precisely, for $t\in[0,T]$ we introduce $f_{1\star j}(t):= \Z_{1\star j}(t)\# f^\init_1$ and use twice Proposition~\ref{prop:norm-W}, each time with $\Z_{1\star j}$ as the smooth (i.e. Lipschitz) curve, to write
 \begin{align*}
 \W_1(f_1(t),f_{1\star j}(t))^2 &\lesssim Q_{\Z_1,\Z_{1\star j}}(t),\\
 \W_1(f_{1\star j}(t),f_{2}(t))^2 &\lesssim \W_1(f_1^\init,f_2^\init)^2 + Q_{\Z_{1\star j},\Z_2}(t),
 \end{align*}
where $\lesssim$ depends only on the $\mathrm{L}^1([0,T];\W^{1,\infty}(\R^3))$ norm of $u_{1\star j}$ and the shared mass of $f^\init_1$ and $f^\init_2$. Note that on the first line, the right-hand side is simpler because $f_1$ and $f_{1\star j}$ share the same initial data. By triangular inequality we have \[\W_1(f_1(t),f_2(t))^2 \lesssim \W_1(f_1(t),f_{1\star j}(t))^2 + \W_1(f_{1_\star j}(t),f_2(t))^2 \lesssim \W_1(f^\init_1,f^\init_2)^2+\widetilde{Q}_j(t),\] so that the proof of the second stability estimate of Theorem~\ref{thm:stabvns} is over once added $\W_1(f_1^\init,f_2^\init)^2$ to both sides of \eqref{ineq:wdeltaD:suite}.
  
  
\section{Proof of Theorem~\ref{thm:BesovVNS}}
\label{sec:proof2}

Let $u^\init \in \dot{\mathbb{B}}_{p,\infty}^{-1+3/p}(\R^3)  \cap \H^s(\R^3)$ for some $s>0$ be a divergence-free vector field. Fix a Leray solution $(f,u)$ associated with the initial condition $(f^\init, u^\init)$. The aim is to prove that there exists $T>0$ small enough such that $u \in \dot{\mathbb{B}}_p(T)$. To this end, we shall start by establishing that for $T>0$ small enough, there exists a unique mild solution $v$ in $\dot{\mathbb{B}}_p(T)$ to the \emph{forced} Navier-Stokes equation
\begin{equation}
\label{eq:ns-forced}
 \partial_t v + \P\big[\textnormal{div}(v\otimes v)\big] - \Delta v = \P F, \quad v|_{t=0} = u^\init,
 \end{equation}
where the (fixed) forcing is precisely the Brinkmann force $F=j_f -\rho_f u$ 
associated with $(f,u)$. Then, by Chemin's uniqueness theorem (namely \cite[Theorem 1.9]{chemincpam}), as $u$ is a Leray solution to~\eqref{eq:ns-forced} on $[0,T]$, $u$ and $v$ must be equal and thus $u \in \dot{\mathbb{B}}_p(T)$. 
With this regularity, recalling Proposition~\ref{prop:Besovwell}, we can apply Theorem~\ref{thm:univns} and deduce the claimed uniqueness result.
 
Everything therefore comes down to the study of \eqref{eq:ns-forced}. We start with a general result for this forced Navier-Stokes equation. The Banach space $\dot{\textnormal{K}}_{p}(T):=\widetilde{\L}^4_T \dot{\B}_{p,\infty}^{-1/2+3/p}(\R^3)$ plays a role similar to $\mathrm{L}^4(0,T;\dot{\H}^1(\R^3))$ in the Kato setting for the Navier-Stokes equation. First, one checks indeed the continuous embedding $\dot{\mathcal{B}}_p(T)\hookrightarrow \dot{\K}_p(T)$. Second, as the $\dot{\K}_p(T)$ norm vanishes when $T\rightarrow 0$ (for elements in the closure of smooth functions), we plan to use the following general result to recover local existence for the VNS system.


  \begin{lem}
  \label{lem:fixedpoint}
  Fix $T>0$ and $p\in(3,\infty)$. Consider $u^\init$ and $G$ such that
  \begin{align}
S(u^\init,G):= e^{t\Delta} u^\init + \int_0^t e^{(t-s)\Delta}G(s)\,\dd s\in \dot{\K}_p(T).
  \end{align}
  There exists $\alpha_p>0$ for which, whenever
    \begin{align}\label{ineq:petite}
\|S(u^\init,G)\|_{\dot{\K}_p(T)} \leq \alpha_p,
    \end{align}
    then there exists a unique $u\in \B_{\dot{\K}_p(T)}(0,2\alpha_p)$ satisfying
    \begin{equation}
    \label{eq:NS-G}
        \partial_t v -\Delta v + \P\big[\textnormal{div}(v\otimes v)\big] = G, \quad u|_{t=0}=u^\init,
            \end{equation}
            in the mild sense. Moreover
            \begin{enumerate}
            \item[$(i)$] If $S(u^\init,G)$ belongs to $\dot{\mathbb{B}}_p(T)$ then so does $v$.
              \item[$(ii)$] If $u^\init$ belongs to $\mathrm{L}^2(\R^3)$ then $v$ is also a Leray solution.
                \end{enumerate}
    \end{lem}
\begin{proof}
A natural strategy is to interpret the mild solution $v$ as the fixed point of the operator 
  \begin{align}\label{eq:fix}
    v \mapsto S(u^\init,G) - \int_0^t e^{(t-s)\Delta} \mathbf{P}\big[\textnormal{div}(v\otimes v)\big]\,\dd s.
  \end{align}
  In order to use \cite[Lemma 5.5]{bcd}, we only need to check that the bilinear map
  \begin{align}\label{map:bil}
\mathfrak{B}:(w,v) \mapsto \int_0^t e^{(t-s)\Delta} \mathbf{P}\big[\textnormal{div}(w\otimes v)\big]\,\dd s,
  \end{align}
  sends continuously $\dot{\K}_p(T)\times \dot{\K}_p(T)$ to $\dot{\K}_p(T)$. Combining the paraproduct estimates given in \cite[Theorem 2.47]{bcd} and \cite[Theorem 2.52]{bcd}, one first checks that the pointwise product sends continuously $\dot{\B}^{-1/2+3/p}_{p,\infty}(\R^3)\times\dot{\B}^{-1/2+3/p}_{p,\infty}(\R^3)$ in $\dot{\B}^{-1+3/p}_{p,\infty}(\R^3)$. Then, for the time variable, a Hölder type estimate implies \begin{align}\label{ineq:contconv}\|\textnormal{div}(w\otimes v)\|_{\widetilde{\L}^2_T\dot{\B}^{-2+3/p}_{p,\infty}(\R^3)} \leq \|w\otimes v\|_{\widetilde{\L}^2_T\dot{\B}^{-1+3/p}_{p,\infty}(\R^3)} \lesssim \|w\|_{\dot{\K}_p(T)}\|v\|_{\dot{\K}_p(T)}.\end{align}
  Now, the expected continuity of the bilinear map $\mathfrak{B}$ defined in \eqref{map:bil} simply follows from estimate (3.39) of \cite[Subsection 3.4.1]{bcd}. Indeed with $(s,r)=(-1+3/p,\infty)$ and $(\rho,\rho_1) = (4,2)$, we get
  \begin{align*}
\|\mathfrak{B}(w,v)\|_{\dot{\K}_p(T)} = \left\| \int_0^t e^{(t-s)\Delta} \mathbf{P}\big[\textnormal{div}(u\otimes v)\big]\,\dd s \right\|_{\dot{\K}_p(T)} \lesssim \|\textnormal{div}(w\otimes v)\|_{\widetilde{\L}^2_T\dot{\B}^{-2+3/p}_{p,\infty}(\R^3)}.
  \end{align*}
With the continuity of the bilinear map $\mathfrak{B}$ at hand, we can directly use \cite[Lemma 5.5]{bcd} to get existence of a mild solution in $\dot{\K}_p(T)$, provided \eqref{ineq:petite} is satisfied. 

\medskip

To prove $(i)$, we only need to check \emph{a posteriori} that our fixed point $v$ satisfies $\mathfrak{B}(v,v) \in\dot{\mathbb{B}}_p(T)$. Recall that by Definition \ref{def:bpTp} this last space is the closure of smooth and rapidly decaying functions for the  $\dot{\mathcal{B}}_p(T) = \widetilde{\L}^\infty_T \dot{\B}_{p,\infty}^{-1+3/p}(\R^3) \cap \widetilde{\L}^1_T \dot{\B}_{p,\infty}^{1+3/p}(\R^3)$ norm. In particular, we infer from the continuous embedding $\dot{\mathcal{B}}_p(T)\hookrightarrow \dot{\K}_p(T)$ the analogous one $\dot{\mathbb{B}}_p(T)  \hookrightarrow \dot{\mathbb{K}}_p(T)$, where $\dot{\mathbb{K}}_p(T)$ is the closure of $\mathscr{C}^\infty([0,T];\mathcal{S}(\R^3))$ in $\dot{\K}_p(T)$. Therefore, the fixed point routine that we have developed above can be implemented inside $\dot{\mathbb{K}}_p(T)$ and our fixed point $v$ actually belongs to that space.

\medskip

 Now, using again estimate (3.39) of \cite[Subsection 3.4.1]{bcd}, with $(s,r,\rho_1)$ as before and $\rho=\infty$ or $\rho=2$, we first have
\begin{align}\label{estim:B1}
  \left\|\mathfrak{B}(w,v)\right\|_{\widetilde{\L}^\infty_T \dot{\B}_{p,\infty}^{-1+3/p}(\R^3)} &\lesssim \|w\|_{\dot{\K}_p(T)} \|v\|_{\dot{\K}_p(T)},\\
\label{estim:B2}  \left\|\mathfrak{B}(w,v)\right\|_{\widetilde{\L}^2_T \dot{\B}_{p,\infty}^{3/p}(\R^3)} &\lesssim \|w\|_{\dot{\K}_p(T)} \|v\|_{\dot{\K}_p(T)}.
\end{align}
Recall that
\[ v = S(u^\init,G) + \mathfrak{B}(v,v).\]
By assumption, $S(u^\init,G)$ belongs to $\dot{\mathcal{B}}_p(T) = \widetilde{\L}^\infty_T \dot{\B}_{p,\infty}^{-1+3/p}(\R^3) \cap \widetilde{\L}^1_T \dot{\B}_{p,\infty}^{1+3/p}(\R^3)$ which embeds continuously in $\widetilde{\L}^2_T \dot{\B}_{p,\infty}^{3/p}(\R^3)$, by interpolation. Using \eqref{estim:B1} -- \eqref{estim:B2}, we see therefore that our fixed point $v$ actually belongs to $\widetilde{\L}_T^\infty \dot{\B}_{p,\infty}^{-1+3/p}(\R^3)$ and $\widetilde{\L}_T^2 \dot{\B}^{3/p}_{p,\infty}(\R^3)$, with the estimate
\[\|v\|_{\widetilde{\L}_T^\infty \dot{\B}_{p,\infty}^{-1+3/p}(\R^3)}+\|v\|_{\widetilde{\L}_T^2 \dot{\B}^{3/p}_{p,\infty}(\R^3)} \lesssim \|S(u^\init,G)\|_{\dot{\mathcal{B}}(T)}+ \|v\|_{\dot{\K}_p(T)}^2.\] 
With this information, using the paraproduct estimates given in \cite[Theorem 2.47]{bcd} and \cite[Theorem 2.52]{bcd}, we get for our fixed point $v$
\begin{align*}
  \|\textnormal{div}(v\otimes v)\|_{\widetilde{\L}^1_T\dot{\B}^{-1+3/p}_{p,\infty}(\R^3)} &\leq \|v\otimes v\|_{\widetilde{\L}^1_T\dot{\B}^{3/p}_{p,\infty}(\R^3)}\\ &\lesssim \|v\|_{\widetilde{\L}_T^2 \dot{\B}^{3/p}_{p,\infty}(\R^3)}^2\\  &\lesssim \|S(u^\init,G)\|_{\dot{\mathcal{B}}_p(T)}^2 + \|v\|_{\dot{\K}_p(T)}^4,\end{align*}
and another use of estimate (3.39) of \cite[Subsection 3.4.1]{bcd} with $\rho=\rho_1=1$ allows eventually to write
\begin{align*}
  \left\|\int_0^t e^{(t-s)\Delta}\mathbf{P}\big[\textnormal{div}(v\otimes v)\big]\,\dd s\right\|_{\widetilde{\L}^1_T \dot{\B}_{p,\infty}^{1+3/p}} &\lesssim \|\textnormal{div}(v\otimes v)\|_{\widetilde{\L}^1_T\dot{\B}^{-1+3/p}_{p,\infty}(\R^3)}\\
    &\lesssim \|S(u^\init,G)\|_{\dot{\mathcal{B}}_p(T)}^2 + \|v\|_{\dot{\K}_p(T)}^4
\end{align*}
Gathering all these estimates, we have proved
\begin{align*}
\|\mathfrak{B}(v,v)\|_{\dot{\mathcal{B}}_p(T)} \leq \|v\|_{\dot\K_p(T)}^2 + \|v\|_{\dot\K_p(T)}^4 + \|S(u^\init,G)\|_{\dot{\mathcal{B}}_p(T)}^2.
\end{align*}
Since we already noticed that $S(u^\init,G)\in\dot{\mathbb{B}}_p(T)\Rightarrow v\in \dot{\mathbb{K}}_p(T)$, this last estimate eventually establishes $(i)$. For $(ii)$, we can proceed just as \cite{lemarie} (where the functional setting is a bit different). One starts with the usual convolution scheme for constrcting Leray solutions to the Navier-Stokes system, with a family of mollifiers $(\rho_\ep)_\ep$. For each $\ep$, the fixed point procedure we just presented in $\dot\K_p(T)$ applies for this modified system (as the smoothed out convection term enjoys the same estimates, uniformly in $\ep$) and gives a (unique) solution $u_\ep \in \B_{\dot\K_p(T)}(0,2\alpha_p)$. Thanks to the usual weak compactness argument for the sequence $(u_\ep)_\ep$, one then recovers a solution which is both in $\dot\K_p(T)$ and Leray, and which must thus coincide with the solution obtained in (i) by uniqueness in $\dot\K_p(T)$.
$\qedhere$ 
\end{proof}
In our case, $G$ is given by the Brinkman force that we interpret as an independent forcing. This  is justified thanks to the following lemma.
 \begin{lem}
 \label{lem:brinkBp}
 The Brinkman force $F:=j_f -\rho_f u$ belongs to $\Ll^2(\R_+;\mathrm{L}^{3/2}(\R^3))$ and $\Ll^1(\R_+;\mathrm{L}^3(\R^3))$, and we have 
     \begin{align}\label{ineq:brinkBeso}
\left\|\int_0^t e^{(t-s)\Delta}\mathbf{P} F(s)\,\dd s\right\|_{\dot{\mathcal{B}}_p(T)} \lesssim \left(\int_0^T \|F(s)\|_{3/2}^2\,\dd s \right)^{1/2} + \int_0^T \|F(s)\|_{3}\,\dd s,
\end{align}
where the constant involved in  $\lesssim$ does not depend on $T$. In particular, there holds
     \begin{align}\label{bel:brink}
       \int_0^t e^{(t-s)\Delta}\mathbf{P} F(s)\,\dd s \in \dot{\mathbb{B}}_p(T),\end{align}
     and
     \begin{align}\label{brink0}
\lim_{T\rightarrow 0^+} \left\| \int_0^t e^{(t-s)\Delta}\mathbf{P} F(s)\,\dd s\right\|_{\dot{\mathcal{B}}_p(T)} = 0. 
     \end{align}
\end{lem}

\begin{proof}
First, let's recall that by Lemma~\ref{lem:reg}, $\rho_f$ and $j_f$ belong to $\Ll^\infty(\R_+;\mathrm{L}^\infty(\R^3)$. Since $u$ is a Leray solution, using the Sobolev embedding $\H^1(\R^3)\hookrightarrow \mathrm{L}^6(\R^3)$ we infer $u\rho_f \in\Ll^2(\R_+;\mathrm{L}^6(\R^3))$. Furthermore $j_f \in \Ll^\infty(\R_+;\mathrm{L}^{3/2}(\R^3))$ (see e.g. \cite[Lemma 4.6]{ehkm}), and thus by interpolation we see that the Brinkman force $F$ belongs to  $\Ll^2(\R_+;\mathrm{L}^6(\R^3))$. Using Cauchy-Schwarz's inequality we have $|F|\leq \rho_f^{1/2}\D(f,u)^{1/2}$ where the dissipation $\D(f,u)$ is defined in \eqref{def:D} so that, the energy inequality \eqref{ineq:nrj} entails $F\in\Ll^2(\R_+;\mathrm{L}^1(\R^3))$. By interpolation, we thus have $F$ in  $\Ll^2(\R_+;\mathrm{L}^q(\R^3))$ for all  $q\in[1,6]$.
By \cite[Lemma 2.4]{bcd}, there holds
  \begin{align*}
    \left\|{\dot{\Delta}}_j \int_0^t e^{(t-s)\Delta} \mathbf{P}F(s) \,\dd s \right\|_p &\leq \int_0^t \| e^{(t-s)\Delta} {\dot{\Delta}}_j \mathbf{P}F(s) \|_p \,\dd s \\
                                                                                &\lesssim \int_0^t e^{-c(t-s) 2^{2j}} \|{\dot{\Delta}}_j \mathbf{P}F(s)\|_p\,\dd s.
  \end{align*}
For all $T>0$, using Young's inequality for the (temporal) convolution we infer
\begin{align}\label{ineq:younghol1}
  \sup_{t\in[0,T]}    \left\|{\dot{\Delta}}_j \int_0^t e^{(t-s)\Delta} \mathbf{P}F(s) \,\dd s \right\|_p &\lesssim \frac{1}{2^j} \left(\int_0^T \|{\dot{\Delta}}_j \mathbf{P}F(s)\|_p^2\,\dd s\right)^{1/2},\\
\label{ineq:younghol2}  \int_0^T   \left\|{\dot{\Delta}}_j \int_0^t e^{(t-s)\Delta} \mathbf{P}F(s) \,\dd s \right\|_p\,\dd t &\lesssim \frac{1}{2^{2j}} \int_0^T \|{\dot{\Delta}}_j \mathbf{P}F(s)\|_p\,\dd s.
  \end{align}
By Bernstein's inequality, the estimates \eqref{ineq:younghol1} -- \eqref{ineq:younghol2} thus imply that
\begin{align*}
  \sup_{t\in[0,T]}    \left\|{\dot{\Delta}}_j \int_0^t e^{(t-s)\Delta} \mathbf{P}F(s) \,\dd s \right\|_p &\lesssim 2^{j(1-\frac{3}{p})} \left(\int_0^T \|{\dot{\Delta}}_j \mathbf{P}F(s)\|_{3/2}^2\,\dd s\right)^{1/2},\\
  \int_0^T   \left\|{\dot{\Delta}}_j \int_0^t e^{(t-s)\Delta} \mathbf{P}F(s) \,\dd s \right\|_p\,\dd t &\lesssim 2^{j(-1-\frac{3}{p})} \int_0^T \|{\dot{\Delta}}_j \mathbf{P}F(s)\|_{3}\,\dd s.
  \end{align*}
 In other words, we have obtained that
  \begin{multline*}
\left\|\int_0^t e^{(t-s)\Delta}\mathbf{P} F(s)\,\dd s\right\|_{\dot{\mathcal{B}}_p(T)} \lesssim \sup_{j\in \mathbf{Z}} \left(\int_0^T \|{\dot{\Delta}}_j\mathbf{P} F(s)\|_{3/2}^2\,\dd s \right)^{1/2} + \sup_{j\in\mathbf{Z}} \int_0^T \|{\dot{\Delta}}_j \mathbf{P}F(s)\|_{3}\,\dd s.
  \end{multline*}
By continuity of the spectral projector ${\dot{\Delta}}_j$ and of the Leray projector on $\mathrm{L}^{3/2}(\R^3)$ and $\mathrm{L}^3(\R^3)$, we have   $\|{\dot{\Delta}}_j \mathbf{P} F(s)\|_{\mathrm{L}^q(\R^3)} \leq \|F(s)\|_{q}$ and we recover \eqref{ineq:brinkBeso}. $F$ can be approximated in $\Ll^2(\R_+;\mathrm{L}^{3/2}(\R^3))$ and $\Ll^1(\R_+;\mathrm{L}^3(\R^3))$ by functions having a compactly supported Fourier transform and this last spectral property is preserved by both $\mathbf{P}$ and the heat flow. The very estimate \eqref{ineq:brinkBeso} that we established shows therefore also the belonging \eqref{bel:brink}. $\qedhere$ 
\end{proof}
       \begin{lem}\label{lem:uinit}
If $u^\init\in\dot{\mathbb{B}}_{p,\infty}^{-1+3/p}$, then $e^{t\Delta}u^\init \in \dot{\mathbb{B}}_p(T)$ and $\lim_{T\to 0^+} \|e^{t\Delta}u^\init\|_{\dot{\K}_{p}(T)} = 0$.
       \end{lem}
       \begin{proof}
         Using \cite[Subsection 5.6.1]{bcd} we have directly $\|e^{t\Delta} u^\init\|_{\dot{\mathcal{B}}_p(T)} \lesssim \|u^\init\|_{\dot{\B}_{p,\infty}^{-1+3/p}(\R^3)}$. Now for all $N\in\N$, we have $S_N e^{t\Delta} u^\init = e^{t\Delta} S_N u^\init$. Having in mind Remark~\ref{rem:trunc}, we infer the implication $u^\init \in\dot{\mathbb{B}}_{p,\infty}^{-1+3/p}(\R^3)\Rightarrow e^{t\Delta} u^\init \in \dot{\mathbb{B}}_p(T)$. We already noted, by interpolation, that $\dot{\mathcal{B}}_p(T) \hookrightarrow \dot{\K}_p(T)$. So for any $N\in\N$ we have
         \begin{align*}
\|e^{t\Delta} u^\init\|_{\dot{\K}_p(T)} \lesssim \|e^{t\Delta} S_N u^\init\|_{\dot{\K}_p(T)} + \|(\textnormal{Id}-S_N)u^\init\|_{\dot{\B}_{p,\infty}^{-1+3/p}(\R^3)}.
         \end{align*}
         But for any fixed $N$, there holds using estimate (3.39) of \cite[Subsection 3.4.1]{bcd} with $\rho=\infty$
         \begin{align*}
           \|e^{t\Delta} S_N u^\init\|_{\dot{\K}_p(T)} &\leq T^{1/4} \|e^{t\Delta} S_N u^\init\|_{\widetilde{\L}^\infty_T \dot{\B}_{p,\infty}^{-1/2+3/p}(\R^3)} \\
                                                       & \lesssim T^{1/4} \|S_N u^\init\|_{\dot{\B}^{-1/2+3/p}(\R^3)_{p,\infty}} \\
           &\lesssim T^{1/4} 2^{N/2} \|u^\init\|_{\dot{\B}^{-1+3/p}_{p,\infty}(\R^3)} \conv{T}{0^+} 0.
         \end{align*}
         The conclusion follows because $(S_N u^\init)_N\rightarrow u^\init$ in $\dot{\B}^{-1+3/p}_{p,\infty}(\R^3)$.
       \end{proof}

       Combining Lemma~\ref{lem:fixedpoint}, Lemma~\ref{lem:brinkBp} and Lemma~\ref{lem:uinit} we recover the existence of a Leray solution belonging to $\dot{\mathbb{B}}_p(T)$ and we may invoke Chemin's uniqueness theorem as explained at the beginning of the proof, so that Theorem~\ref{thm:BesovVNS} is proven.

       \section{Proof of Theorem~\ref{thm:BesovVNS:glob}}
\label{sec:proof3}
        Assume additionally that $u^\init \in {\dot{\B}^{-3/2}_{2,\infty}}(\R^3)$. The aim is now to prove that the local in time $\dot{\mathbb{B}}_p(T)$ solution, \emph{a priori} defined on a small time intervall $(0,T)$, can be globally extended up to a smallness condition like \eqref{hyp:smallnessBesov}. As a matter of fact, the argument we present below is similar to the one used by Danchin \cite{Danchin} to obtain a unique Fujita-Kato solution to VNS -- that is for $u^\init \in {\dot{\H}^{1/2}}(\R^3)$, see \cite[Theorem 1.6]{Danchin}.
\subsection{Decay results for the VNS system}
Our proof of Theorem~\ref{thm:BesovVNS:glob}  relies on \cite{Danchin}, from which we extract the following statement.  
\begin{thm}[\cite{Danchin}]\label{thm:danchin}
  Consider an admissible initial data $(f^\init,u^\init)$ in the sense of Definition~\ref{def:admi} and assume furthermore that $u^\init\in \dot{\B}^{-3/2}_{2,\infty}(\R^3)\cap \H^1(\R^3)$ and $N_q(f^\init)< \infty$ for some $q>5$. There exists a nondecreasing function $\Theta$ continuously vanishing at $0$ and a nonincreasing function $\Psi$, such that the following holds. Whenever
  \begin{align}\label{ineq:danch1}
\|u^\init\|_{\H^1(\R^3)} + M_2 f^\init \leq\Psi\big(  \| f^\init\|_{1} +  N_q(f^\init) +\|u^\init\|_{\dot{\B}_{2,\infty}^{-3/2}(\R^3)}\big),
  \end{align}
  then the  \textup{VNS} system admits a global Leray solution $(f,u)$ for which the associated moments $j_f$, $\rho_f$ and the Brinkman force $F:= j_f-\rho_f u$ satisfy the following bounds and decay estimates
  \begin{multline}\label{eq:decay}
\sup_{t\geq 0} \Big\{\|F(t)\|_{\infty}+ \|\rho_f(t)\|_\infty + \|j_f(t)\|_\infty +  t^{5/4}\|F(t)\|_1\Big\}  \\\leq \Theta(\|u^\init\|_{\H^1(\R^3)} + M_2 f^\init + \|f^\init\|_{\mathrm{L}^1} +  N_q(f^\init) +\|u^\init\|_{\dot{\B}_{2,\infty}^{-3/2}(\R^3)}\big),
  \end{multline}
  and also
  \begin{multline}\label{eq:decay2}
\int_0^{+\infty} t^{9/4} \|F(t)\|_2^2\,\dd t  \leq \Theta(\|u^\init\|_{\H^1(\R^3)} + M_2 f^\init + \|f^\init\|_{\mathrm{L}^1} +  N_q(f^\init) +\|u^\init\|_{\dot{\B}_{2,\infty}^{-3/2}(\R^3)}\big).
  \end{multline}
\end{thm}
 Of course, the solution built in \cite{Danchin} is actually more regular than a simple Leray solution and his main result includes also well-posedness in the associated functional setting. However, as far as our proof is concerned, we only need here this existence result. Theorem~\ref{thm:danchin} collects several estimates obtained in \cite{Danchin}, let us comment on where they appear in \cite{Danchin}. The smallness condition \eqref{ineq:danch1} is inherited from (1.2) in \cite[Theorem 1.1]{Danchin}. In the statement and proof of \cite[Theorem 1.1]{Danchin}, all the constants used to establish the estimates depend on the initial data, which explains the form of r.h.s. of \eqref{eq:decay} and \eqref{eq:decay2}.
 
 Also note that, since $q>5$, using $\|m_0 f^\init\|_\infty + \|m_2 f^\init\|_\infty \lesssim N_q(f^\init)$, the $\mathrm{L}^\infty(\R^3)$ norms of $m_0 f^\init$ and $m_2 f^\init$ (see the dependence of $c_0$ in \cite[Theorem 1.1]{Danchin}) are here simply replaced by $N_q(f^\init)$. 
 
 Moreover, estimate \cite[(1.2)]{Danchin} involves the $L^1$ norm of the initial fluid velocity, but as indicated in \cite[Remark 1.3]{Danchin}, it suffices to assume that $u^\init \in \dot{\B}_{2,\infty}^{-3/2}(\R^3)$ and to replace all dependencies on $\|u^\init\|_{L^1(\R^3)}$ by dependencies on $\|u^\init\|_{\dot{\B}_{2,\infty}^{-3/2}(\R^3)}$.

 The $\lesssim t^{-5/4}$ decay for the $\mathrm{L}^1(\R^3)$ of the Brinkman force is given in \cite[Corollary 1.4]{Danchin}. The bound in $\mathrm{L}^\infty(\R^3)$ for the Brinkman force comes from the one of $j_f$ and $\rho_f$, which in turn is obtained because the solution $(f,u)$ built by Danchin \cite{Danchin} actually satisfies $\nabla u \in\mathrm{L}^1(\R_+;\mathrm{L}^\infty(\R^3))$ (as stated in \cite[Theorem 1.1]{Danchin}). As already noticed in \cite{hkm2}, an adequate change of variable (originally introduced in \cite{bardegond}) allows then to control along the time the $\mathrm{L}^\infty(\R^3)$ norm of kinetic moments of $f$ by the initial data. We refer to the computation in \cite[Appendix A]{Danchin}. Lastly, in order to obtain estimate \eqref{eq:decay2},  
 we observe that \cite[Estimate (3.15)]{Danchin} implies a control on 
 \begin{equation*}
     \int_0^\infty t^\beta D_1(t)\,\dd t,
 \end{equation*}  
 in terms of the initial data, for all $\beta\in (1,5/2)$, where the quantity $D_1(t)$ defined in \cite[(2.5)]{Danchin} satisfies
 $$D_1(t)\geq \frac{1}{2}\int_{\R^3\times\R^3} f(t)|v-u(t)|^2\,\dd x\,\dd v.$$
 On the other hand, by Cauchy-Schwarz inequality we have
   \begin{align}\label{ineq:brinkman-L2}
   \|F(t)\|_2^2 
   &\leq  \|\rho_f(t)\|_\infty \int_{\R^3\times\R^3} f(t)|v-u(t)|^2\,\dd x\,\dd v.
   \end{align}
 Thus finally,
   \begin{align*}t^{9/4}\|F(t)\|_2^2 
   &\leq 
   2 \sup_{t\geq 0} \|\rho_f(t)\|_\infty (t^{9/4}D^1(t))
   \end{align*} is controlled in $L^1(\R_+)$ in terms of the initial data.     As a matter of fact, in addition to the previous estimates, several other decay estimates are obtained in the course of the proof of \cite[Theorem 1.1]{Danchin} and apply as well for our solutions.

     \subsection{Instantaneous regularization and propagation of initial norms}

Our first step is to let the system evolve for a short amount of time, so that the fluid velocity becomes smoother in order to use Theorem~\ref{thm:danchin} which asks for an initial fluid velocity in $\H^1(\R^3)$. We fix $T>0$ for which we have a solution in $\dot{\mathbb{B}}_p(T)$ for the VNS system; without loss of generality, we assume $T\leq 1$.

     \medskip

     Combining Lemma \ref{lem:reg} (ii) and \eqref{ineq:nrj} together with \eqref{ineq:brinkman-L2}, we observe that for any Leray solution such that $M_6f^\init$ is finite, 
     the Brinkman force is in $\Ll^2(\R_+;\mathrm{L}^2(\R^3))$. Thus any such Leray solution is instantaneously regularized and belongs (at least) to $\mathscr{C}^0(\R_{+}^\star;\H^1(\R^3))$. Owing to the energy-dissipation estimate \eqref{ineq:nrj}, there exists $t^\ast\in (0,T)$ for which 
     \begin{align}\label{ineq:H1}\|u(t^\star)\|_{\H^1(\R^3)}^2 \leq \frac{3}{T}\E^\init.\end{align}
     Our goal is now to replace the initial data $(f^\init,u^\init)$ by $(f^\star,u^\star) = (f(t^\star),u(t^\star))$ and try to prove that under a smallness condition like \eqref{hyp:smallnessBesov} on $(f^\init,u^\init)$, the smallness condition \eqref{ineq:danch1} is indeed satisfied for the shifted initial data $(f^\star,u^\star)$, that is :
       \begin{align}\label{ineq:todanch}
       \|u^\star\|_{\H^1(\R^3)} + M_2 f^\star \leq\Psi\big(  \| f^\star\|_{1} +  N_q(f^\star) +\|u^\star\|_{\dot{\B}_{2,\infty}^{-3/2}(\R^3)}\big).
         \end{align}
         If  we manage to do so, then we'll be in position to use the decay estimates \eqref{eq:decay} -- \eqref{eq:decay2} of Danchin's Theorem~\ref{thm:danchin}. 
         
       First of all, using \eqref{ineq:H1} and the energy-dissipation estimate \eqref{ineq:nrj}, and observing that $T$ is a decreasing function of $\E^\init$, as can be seen in the proofs of Lemma \ref{lem:fixedpoint} and Lemma \ref{lem:uinit}, we have  for some non-decreasing function $\alpha$ that
       \begin{equation}\label{ineq:toprove}
\|u^\star\|_{\H^1(\R^3)} + M_2 f^\star  \leq \alpha\big(\E^\init). 
\end{equation}
Let us now study the quantity $\| f^\star\|_{1} +  N_q(f^\star) +\|u^\star\|_{\dot{\B}_{2,\infty}^{-3/2}(\R^3)}$. To begin with, by conservation of the mass for the Vlasov equation, we have 
$$
\|f^\star\|_1 = \|f^\init\|_1.
$$ 
Furthermore, thanks to Lemma~\ref{lem:reg}, we know that the fluid velocity $u$ belongs to $\Ll^1(\R_+;\mathrm{L}^\infty(\R^3))$ and a careful inspection of the estimates shows that this control on a time interval $T$ depends increasingly on $T$ and $\E^\init$. Owing to \cite[Lemma 4.6]{hkm2}, we therefore have (as $t^\star\leq 1$) the estimate 
$$N_q(f^\star) \leq \gamma_1(\E^\init + N_q(f^\init)),
$$ 
for yet another non-decreasing function $\gamma_1$. Our last duty is to propagate the initial $\dot{\B}^{-3/2}_{2,\infty}(\R^3)$ seminorm in (finite) time: we rely on a simplified version of some the arguments proposed by Danchin in \cite[Section 3, Step 3]{Danchin}. Thanks to the so-called Besov maximal regularity estimate (3.39) of \cite[Subsection 3.4.1]{bcd} that we use here for $(\rho,s,p,r)=(\infty,-3/2,2,\infty)$ together with the embedding $\mathrm{L}^1(0,T;\mathrm{L}^1(\R^3))\hookrightarrow \widetilde{\L}^1_T \dot{\B}^{-3/2}_{2,\infty}(\R^3)$, we first infer
\begin{align*}
\|u(t^\star)\|_{\dot{\B}^{-3/2}_{2,\infty}(\R^3)} \lesssim \|u^\init\|_{\dot{\B}^{-3/2}_{2,\infty}(\R^3)} + \int_0^{t^\star} \|(u\cdot\nabla) u(s)\|_{1}\,\dd s + \int_0^{t^\star} \|F(s)\|_{1}\,\dd s,
\end{align*}
where $F= j_f-\rho_f u $ is the Brinkman force.  Thanks to Cauchy-Schwarz's inequality we have the pointwise estimate $|F|\leq \rho_f^{1/2} \D(u,f)^{1/2}$ from which we infer, using the energy-dissipation formula \eqref{ineq:nrj} and that $t^\star\leq 1$ \[\int_0^{t^\star}\|F(s)\|_{1}\,\dd s \leq \left(\int_0^{t^\star} \|\rho_f(s)\|_1\,\dd s \right)^{1/2} \E(u^\init,f^\init)^{1/2} \leq  \|f^\init\|_1^{1/2} \E(u^\init,f^\init)^{1/2}.\]
On the other hand, by Cauchy-Schwarz's inequality and the energy dissipation inequality 
\begin{align*}
\int_0^{t^\star} \|(u\cdot\nabla) u(s)\|_1 \,\dd s \leq \left(\int_0^{t^\star} \|u(s)\|_2^2\,\dd s\right)^{1/2} \left(\int_0^{t^\star}\|\nabla u(s)\|_2^2\,\dd s\right)^{1/2}
  \leq \sqrt{2} \E(u^\init,f^\init). 
\end{align*}
Gathering those estimates, we have established for some non-decreasing function $\gamma_2$
\[\|u(t^\star)\|_{\dot{\B}_{2,\infty}^{-3/2}(\R^3)} \leq \gamma_2(\|u^\init\|_{\dot{\B}^{-3/2}_{2,\infty}(\R^3)} +\|f^\init\|_1 + \E^\init).\]
All in all, considering the non-decreasing function $\gamma = \max(\gamma_1,\gamma_2)$ we have proved,
\begin{equation}
    \label{eq:starstar}
\| f^\star\|_{1} +  N_q(f^\star) +\|u^\star\|_{\dot{\B}_{2,\infty}^{-3/2}(\R^3)} \leq \gamma (\|u^\init\|_{\dot{\B}^{-3/2}_{2,\infty}(\R^3)}+\|f^\init\|_1 +  N_q(f^\init)).
\end{equation}


\subsection{Ensuring Danchin's decay estimates}
With the estimates \eqref{ineq:toprove} and \eqref{eq:starstar} at hand, as $\Psi$ is non-increasing, in order to ensure~\eqref{ineq:todanch}, it is enough to impose that
\begin{align}\label{road2Danch}
\alpha(\E^\init) \leq \Psi \circ \gamma(\|u^\init\|_{\dot{\B}^{-3/2}_{2,\infty}(\R^3)}+\|f^\init\|_1 + N_q(f^\init)).
\end{align}

In particular, if this smallness condition is satisfied, we're in position to use Danchin's decay estimates \eqref{eq:decay} -- \eqref{eq:decay2}.
\subsection{Global fixed-point strategy}
Now we intend to follow the same fixed-point procedure for the proof of Theorem~\ref{thm:BesovVNS} (see Lemma~\ref{lem:fixedpoint}), arguing that for small enough initial data such a procedure can be applied on $\dot{\K}_p(+\infty)$. 
Since $\dot{\mathcal{B}}_p(T)\hookrightarrow \dot{\K}_p(T)$ for all times (with a constant independent of $T$), the estimate  (see \cite[Subsection 5.6.1]{bcd}) $\|e^{t\Delta} u^\init\|_{\dot{\mathcal{B}}_p(+\infty)} \lesssim \|u^\init\|_{\dot{\B}_{p,\infty}^{-1+3/p}(\R^3)}$ ensures that the ``initial condition'' part of the smallness condition \eqref{ineq:petite} can indeed be eternally ensured upon smallness of the initial data. More precisely, for some non-decreasing (in fact, linear) function $\Phi_2$ we have
\begin{align}\label{eq:phi2}
 \|e^{t\Delta} u^\init\|_{\dot{\mathcal{B}}_p(+\infty)} \leq  \frac{\alpha_p}{2} \Phi_2(\|u^\init\|_{\dot{\B}_{p,\infty}^{-1+3/p}(\R^3)}),
\end{align}
where $\alpha_p$ is the threshold given in Lemma~\ref{lem:fixedpoint}. 
As we already have a solution in $\dot{\mathcal{B}}_p(t^\star)$, we intent to concatenate a solution on $[t^\star,+\infty)$ to recover a solution in $\dot{\mathcal{B}}_p(+\infty)$. The heart of the matter is therefore to control the size, on arbitrary large interval of times, of the (heat flow generated by the) Brinkman force, after the regularization time lap $[0,t^\star]$, and we only need to prove
\begin{align}\label{toprove}
\left\|\int_{t^\star}^{t} e^{(t-s)\Delta} \mathbf{P} F(s)\,\dd s \right\|_{\dot{\mathcal{B}}_p(+\infty)} \leq \frac{\alpha_p}{2},
\end{align}
where $\alpha_p$ is the threshold given in Lemma~\ref{lem:fixedpoint}. Let us check that such a smallness condition can indeed be fulfilled solely upon smallness for the initial data, in the sense of \eqref{hyp:smallnessBesov}. For consiceness, let's introduce the following notations
\begin{align*}
  \C^\init &:= \| u^\init\|_{\dot{\B}_{p,\infty}^{-1+3/p}(\R^3)}  + \E^\init + \| f^\init\|_{1} +  N_q(f^\init) +\|u^\init\|_{\dot{\B}_{2,\infty}^{-3/2}(\R^3)},\\
  \C^\star &:= \|u^\star\|_{\H^1(\R^3)} + M_2 f^\star + \|f^\star\|_{\mathrm{L}^1} +  N_q(f^\star) +\|u^\star\|_{\dot{\B}_{2,\infty}^{-3/2}(\R^3)},             
\end{align*}
so that $\C^\init$ refers to the smallness condition \eqref{hyp:smallnessBesov} while $\C^\star$ refers to the smallness condition \eqref{ineq:danch1} (shifted at time $t^\star$) and Danchin's decay estimates \eqref{eq:decay} -- \eqref{eq:decay2}. With this notation, estimates \eqref{ineq:toprove} and \eqref{eq:starstar} that we proved above rephrases as \begin{align}\label{ineq:tau}\C^\star \leq \tau(\C^\init),\end{align}
where $\tau=\max(\alpha,\gamma)$. Our strategy boils down to prove the following estimate for some non-decreasing function $\Phi_3$
\begin{align}\label{eq:inceta}
\left\|\int_{t^\star}^{t} e^{(t-s)\Delta} \mathbf{P} F(s)\,\dd s \right\|_{\dot{\mathcal{B}}_p(T)}  \leq \frac{\alpha_p}{2} \Phi_3 (\C^\init),
\end{align}
independently of $T$. Using estimate \eqref{ineq:brinkBeso} of Lemma~\ref{lem:brinkBp} shifting the initial time to $t^\star$, we have  
    \begin{align}\label{ineq:heatF}
\left\|\int_{t^\star}^t e^{(t-s)\Delta}\mathbf{P} F(s)\,\dd s\right\|_{\dot{\mathcal{B}}_p(T)} \lesssim \left(\int_{t^\star}^T \|F(s)\|_{3/2}^2\,\dd s \right)^{1/2} + \int_{t_0}^T \|F(s)\|_{3}\,\dd s,
    \end{align}
    where $\lesssim$ does not depend on $T$. On the one hand, we have by H\"older's inequality, using \eqref{eq:decay} (shifting the initial time to $t^\star$) and \eqref{ineq:tau}
 \begin{align*}
\int_{t^\star}^T \|F(s)\|_{3/2}^2\,\dd s  &\leq  \int_{t^\star}^T \|F(s)\|_{1}^{4/3}  \|F(s)\|_{\infty}^{2/3} \,\dd s   \\
                                          &\leq  \Theta(\C^\star)^2 \int_{t^\star}^{+\infty} s^{-5/3} \,\dd s\\
                                          &\leq \Theta(\tau(\C^\init))^2 \int_{t^\star}^{+\infty} s^{-5/3}\,\dd s\\
   &=\eta_1(\C^\init),
 \end{align*}
 where $\eta_1$ is a non-decreasing function. On the other hand, combining \eqref{eq:decay} and \eqref{eq:decay2} with Hölder's inequality we infer 
  \begin{align*}
\int_{t^\star}^T \|F(s)\|_{3}\,\dd s  &\leq  \int_{t^\star}^T \|F(s)\|_{2}^{2/3}  \|F(s)\|_{\infty}^{1/3} \,\dd s   \\
                                      & =  \int_{t^\star}^T \big(t^{9/4} \|F(s)\|_2^2\big)^{1/3} t^{-3/4} \|F(s)\|_\infty^{1/3}\,\dd s \\
                                      &\leq \Theta(\C^\init)^{1/3} \left(\int_{t^\star}^T t^{9/4} \|F(s)\|_2^2\,\dd s\right)^{1/3} \left( \int_{t^\star}^T t^{-9/8}\,\dd t\right)^{2/3}
    \\
    &\leq \Theta(\C^\init)^{2/3} \left( \int_{t^\star}^{+\infty} t^{-9/8}\,\dd t\right)^{2/3} = \eta_2(\C^\init).
   \end{align*}
   Using the estimates involving $\eta_1$ and $\eta_2$ inside \eqref{ineq:heatF}, we recover the expected estimate \eqref{eq:inceta} for some non-decreasing function $\Phi_3$. The choice of $\Phi:=\max(\Phi_1,\Phi_2,\Phi_3)$ allows to ensure, under the smallness condition \eqref{hyp:smallnessBesov} that Danchin's decay estimate is satisfied thanks to \eqref{road2Danch} and we therefore can apply the fixed-point procedure on $[t^\star,+\infty)$.
   
 \appendix

 \section{Littlewood-Paley theory}\label{sec:lp}
 \subsection{Dyadic partition of unity}\label{subsec:dyadic}
 We recall (see \cite[Proposition 2.10]{bcd} for instance) the existence of smooth radial functions $\chi,\ffi:\R^3\rightarrow[0,1]$ with $\chi$ supported on a vicinity of the unit ball, equal to $1$ on it, and $\ffi$ supported on a vicinity of the unit sphere, equal to $1$ on it, and vanishing near the origin, such that the following holds. Defining for $j\in\mathbf{Z}$, $\ffi_j(\xi) := \ffi(\xi/2^j)$, we have $\ffi_j\ffi_\ell =0$ as soon as $|j-\ell|\geq 2$ and the following identities
 \begin{align}\label{eq:dyapart1}
   \chi + \sum_{j\in\mathbf{N}} \ffi_j &= 1,\text{ on }\R^3,\\
\label{eq:dyapart2}   \sum_{j \in\mathbf{Z}} \ffi_j &= 1,\text{ on }\R^3 \setminus\{0\}.
 \end{align}
 \subsection{Localization}\label{subsec:loc}
 As the Fourier transform $\mathcal{F}$ is an automorphism on the space of tempered distributions $\mathcal{S}'(\R^3)$, one can define for $j\in\mathbf{Z}$ the (homogeneous) localization operator $\dot{\Delta}_j$ by means of its action on the ``Fourier side'', that is for any tempered distribution $u\in\mathcal{S}'(\R^3)$:
 \begin{align*}
\dot{\Delta}_j u = \mathcal{F}^{-1}( \ffi_j \widehat{u}).
\end{align*}
Then, as customary, we define the (non-homogeneous) localization operators in the following way. For $j\geq 0$ we simply set $\Delta_j = \dot{\Delta}_j$ and for $j\leq -2$ we set $\Delta_j =0$. Lastly, we set for any tempered distribution $u$, $\Delta_{-1} u := \mathcal{F}^{-1}(\chi \widehat{u})$. Of course, if $\theta,\psi\in\mathcal{S}(\R^3)$ are such that $\widehat{\theta} = \ffi$ and $\widehat{\psi} = \chi$, then $\dot{\Delta}_j$ is simply the convolution operator with $x\mapsto 8^j \theta(2^j x)$ while $\Delta_{-1}$ is the convolution opperator with $\psi$. In particular, it can be readily checked that defining $S_j:=\sum_{k\leq j-1} \Delta_k$, the sequence $(S_j)_{j\geq 0}$ corresponds to a usual mollifier, that is the convolution with $x\mapsto 8^j\psi(2^j x)$. On the ``Fourier side'', this simply amounts to multiply by $\xi\mapsto \chi(\xi/2^j)$ which is compactly supported and is equal to $1$ on a vicinity of the ball of size $2^j$.

\begin{rem}\label{rem:obs}
A fundamental observation is the following. Because of \eqref{eq:dyapart1} and the previous definition of the dyadic blocks $\Delta_j$, one actually has the identity $\sum_{j\in\mathbf{Z}} \Delta_j = \textnormal{Id}_{\mathcal{S}'(\R^3)}$, where the convergence holds pointwisely in $\mathcal{S}'(\R^3)$. Note however that in \eqref{eq:dyapart2}, part of the spectral information is missing as the partition covers only the blunted space $\R^3\setminus\{0\}$. In particular, the equality $\sum_{j\in\mathbf{Z}} \dot{\Delta}_j = \textnormal{Id}_{\mathcal{S}'(\R^3)}$ holds only modulo polynomials (as these are the only tempered distributions whose Fourier transform is localized at the origin).
\end{rem}

\subsection{Bernstein estimates}

The following well-known Bernstein inequality play a crucial proof (for a simple proof, see for instance \cite[Lemma 2.1]{bcd}).
\begin{lem}\label{lem:bernstein}
Fix $1\leq q \leq p \leq \infty $. For any $u\in\mathcal{S}'(\R^3)$ and $j\in\mathbf{Z}$ there holds
  \begin{align*}
 \|\dot{\Delta}_j u\|_{\mathrm{L}^p(\R^3)} \lesssim 2^{j\left(\frac{3}{q}-\frac{3}{p}\right)} \|\dot{\Delta}_j u\|_{\mathrm{L}^q(\R^3)},
  \end{align*}
  where the constant behind $\lesssim$ depends only on the function $\ffi$ introduced in Subsection~\ref{subsec:dyadic}.
\end{lem}

\subsection{Besov spaces}
We first introduce the following definition.
  \begin{defi}
For $s \in \R$, $p,q \in [1,\infty]$ and any tempered distribution $u\in\mathcal{S}'(\R^3)$, we define the \textsf{non homogeneous Besov norm} as
  \[
 \|u \|_{\B^s_{p,q}(\R^3)} = \| 2^{js} \| \Delta_j u\|_{\mathrm{L}^p(\R^3)} \|_{\ell^q(\mathbf{Z})},
\]
and the \textsf{homogeneous Besov semi-norm} as 
  \[
  \|u \|_{\dot{\B}^s_{p,q}(\R^3)} = \| 2^{js} \| \dot{\Delta}_j u\|_{\mathrm{L}^p(\R^3)} \|_{\ell^q(\mathbf{Z})}.
  \]
\end{defi}
With this definition at hand, we can at least define the space $\B^s_{p,q}(\R^3)$ as the space of tempered distribution having a finite non homogeneous Besov norm (and it is a Banach space). For the homogeneous case, one has to be cautious as we only have a semi-norm. This distinction is of course directly linked to Remark~\ref{rem:obs} of Subsection~\ref{subsec:loc}. The motivates, following \cite{bcd}, the introduction of the following subspace of \emph{homogeneous tempered distributions} \[\mathcal{S}_h(\mathbf{R}^3) := \textnormal{Ker}\Big(\textnormal{Id}_{\mathcal{S}'(\R^3)}-\sum_{j\in\mathbf{Z}}\dot{\Delta}_j\Big).\]
  \begin{defi}\label{def:homobeso}
    The \textsf{homogeneous Besov} space $\dot{\B}_{p,q}^s(\R^3)$ consists of all elements $u\in \mathcal{S}'(\mathbf{R}^3)$ for which $\|u\|_{\dot{\B}_{p,q}^s(\R^3)}<+\infty$. 
  \end{defi}
  It can be checked that when $s<3/p$ or $(s,q)=(3/p,1)$, the space $\dot{\B}_{p,q}^s(\mathbf{R}^3)$ equipped with this norm is a Banach space (see \cite[Theorem 2.25]{bcd} or \cite[Chapter 3]{LR2002}).
  
  \begin{prop}\label{prop:besnegborn}
    For $p <\infty$ and $r\in[1,\infty]$, any element $u$ of $\mathrm{L}^r(\R^3)\cap \dot{\B}_{p,1}^{3/p}(\R^3)$ is bounded on $\R^3$ and the following estimate holds
    \[\|u\|_\infty \lesssim \|u\|_r + \|u\|_{\dot{\B}_{p,1}^{3/p}(\R^3)}.\]
\end{prop}
\begin{proof}
We have $u=\sum_{j\in\mathbf{Z}} \Delta_j u$, or equivalently $u = \psi \star u + \sum_{j\in\N} \dot{\Delta}_j u$ (see Subsection~\ref{subsec:loc} for the definition of $\psi$ and these identities). Since $\chi$ is a Schwartz function we have by Hölder's inequality $\|\psi \star u\|_\infty \lesssim \|u\|_r$, so we focus on the sum. Using Bernstein's inequality of Lemma~\ref{lem:bernstein} we have \[\|\Delta_j u\|_{\mathrm{L}^\infty(\R^3)} \lesssim 2^{j\frac{3}{p}} \|\Delta_j u \|_{\mathrm{L}^q(\R^3)},\]
and by definition of the $\dot{\B}_{p,1}^{3/p}(\R^3)$ semi-norm, we deduce
\begin{align*}
\sum_{j\in\N} \|\dot{\Delta}_j u\|_{\mathrm{L}^\infty(\R^3)} \lesssim \|u\|_{\dot{\B}^{3/p}_{p,1}(\R^3)},
\end{align*}
so that the sum we're looking at is actually normally convergent in $\mathrm{L}^\infty(\R^3)$. $\qedhere$
\end{proof}
\subsection{Chemin-Lerner time-space Besov spaces}\label{subsec:chemlen}
For time-depending functions, Chemin and Lerner introduced in \cite{cl} a class of spaces in which time integration is performed \emph{before} the discrete $\ell^q(\mathbf{Z})$ norm involved in the definition of the Besov norm. 

\medskip

More precisely, fix $T>0$, $I:=(0,T)$ and consider the subspace $\mathcal{S}_{h,T}'(\mathbf{R}^3)\subset \mathcal{S}'(\mathbf{R}\times\R^3)$ that consists of distributions $u$ satisfying the following identity
\begin{align*}
    u = \sum_{j\in\mathbf{Z}} \dot{\Delta}_j u,\text{ in }\mathscr{D}'(I\times\R^3), 
\end{align*}
where we straightforwardly extended the definition of the spatial localization operators $\dot{\Delta}_j$ to $\mathcal{S}'(\R\times\R^3)$.
In the context of homogeneous Besov space, we have then the following definition.
  \begin{defi}
Fix $T>0$, $s \in \R$, $r,p,q \in [1,\infty]$. The space $\widetilde{\L}^r_T \dot{\B}^s_{p,q}(\mathbf{R}^3)$ is the subspace of $\mathcal{S}_{h,T}(\R^3)$ whose elements $u$ satisfy 
  \[
 \|u\|_{\widetilde{\L}_T^r \dot{\B}_{p,q}^s(\R^3)} := \| 2^{js} \| \dot{\Delta}_j u\|_{\mathrm{L}^r(0,T;\mathrm{L}^p(\R^3))} \|_{\ell^q(\mathbf{Z})}<\infty.
\]
\end{defi}
Just as for the usual homogeneous Besov spaces, equipped with this norm and considered as a subspace of $\mathscr{D}'(I\times\R^3)$, the space $\widetilde{\L}^r_T \dot{\B}^s_{p,q}(\mathbf{R}^3)$ is a Banach space as soon as $s<3/p$ or $(s,q)=(3/p,1)$. Of course, a perfectly similar definition holds for $\widetilde{\L}^r_T\B^s_{p,q}(\R^3)$ in the non-homogeneous case. 

\medskip

As explained in \cite[Section 5.6]{bcd}, the heat flow is naturally related to the following homogeneous spaces, for $p>3$ \begin{align}\label{def:BpT}
\dot{\mathcal{B}}_p(T):=\widetilde{\L}^\infty_T\dot{\B}^{-1+3/p}_{p,\infty}(\R^3)\cap \widetilde{\L}^1_T\dot{\B}^{1+3/p}_{p,\infty}(\R^3),\end{align}
which is equipped with the norm 
\begin{align}\label{def:BpTnorm}
\|\cdot\|_{\dot{\mathcal{B}}_p(T)}:=\|\cdot\|_{\widetilde{\L}^\infty_T\dot{\B}^{-1+3/p}_{p,\infty}(\R^3)} + \|\cdot\|_{\widetilde{\L}^1_T\dot{\B}^{1+3/p}_{p,\infty}(\R^3)}.\end{align}
The following proposition will be useful. 
\begin{prop}\label{prop:useful}
For $T>0$ and $p>3$, we have the two continuous embeddings $\dot{\mathcal{B}}_p(T)\hookrightarrow \widetilde{\L}^2_T\dot{\B}_{p,\infty}^{3/p}(\R^3)$  and $\dot{\mathcal{B}}_p(T)\cap \mathrm{L}^\infty(0,T;\mathrm{L}^2(\R^3)) \hookrightarrow \mathrm{L}^\infty(0,T;\B_{p,\infty}^{-1+3/p}(\R^3))$.
\end{prop}
\begin{proof}
The first embedding is obtained by interpolation. For the second one, we first note that $\widetilde{\L}_T^\infty\dot{\B}_{p,\infty}^{-1+3/p}(\R^3) = \mathrm{L}^\infty(0,T;\dot{\B}^{-1+3/p}_{p,\infty}(\R^3))$, so that we actually only need to prove the continuous embedding $\dot{\B}^{-1+3/p}_{p,\infty}(\R^3)\cap\mathrm{L}^2(\R^3)\hookrightarrow \B_{p,\infty}^{-1+3/p}(\R^3)$, which boils down to the use of Young's inequality to check that $\Delta_{-1}$ maps continuously $\mathrm{L}^2(\R^3)$ to $\mathrm{L}^p(\R^3)$. $\qedhere$
\end{proof}

For these spaces where the third index is infinite, functions having a spectral decay at infinity (like Schwartz functions for instance) are not dense in general. A narrower space introduced by Chemin in \cite{chemincpam,cheminJAM} consists precisely in taking the closure of smooth functions for this norm. Following the notations of \cite{barker}, we end this paragraph with the following definitions.
\begin{defi}\label{def:bpp}
For $p\in(3,\infty)$, $\dot{\mathbb{B}}_{p,\infty}^{-1+3/p}(\R^3)$ is the closure of $\mathcal{S}(\R^3)$ for the $\dot{\B}_{p,\infty}^{-1+3/p}(\R^3)$ norm. 
\end{defi}
The important difference between $\dot{\B}_{p,\infty}^{-1+3/p}$ and $\dot{\mathbb{B}}_{p,\infty}^{-1+3/p}(\R^3)$   is that elements of the latter satisfy  
\begin{equation}\label{eq:propB}
\lim_{j \to +\infty}  2^{j(-1+3/p)} \| \Delta_j u \|_{p} = 0.
\end{equation}
\begin{defi}\label{def:bpTp}
For $T>0$ and $p>3$, $\dot{\mathbb{B}}_p(T)$ is the closure of $\mathscr{C}^\infty([0,T];\mathcal{S}(\R^3))$ for the $\dot{\mathcal{B}}_p(T)$ norm defined in \eqref{def:BpTnorm}.
\end{defi}
Again, the core difference between $\dot{\mathcal{B}}_p(T)$ and $\dot{\mathbb{B}}_p(T)$ is that elements of the latter satisfy  
\begin{equation}
\label{eq:propBp}
\lim_{j \to +\infty} \left( 2^{j(-1+3/p)} \| \Delta_j u \|_{\mathrm{L}^\infty(0,T; \mathrm{L}^p(\R^3))} + 2^{j(1+3/p)} \| \Delta_j u \|_{\mathrm{L}^1(0,T; \mathrm{L}^p(\R^3))}  \right)=0.
\end{equation}
\begin{rem}\label{rem:trunc}
  Be it for $\dot{\mathbb{B}}_{p,\infty}^{-1+3/p}(\R^3)$ or $\dot{\mathbb{B}}_p(T)$, it is equivalent to ask an approximation by functions having a compactly supported Fourier transform. In particular, on each of these spaces, the spectral projections $(S_j)_j$ converge pointwisely to the identity map. 
\end{rem}

  \subsection{Relevancy of $\dot{\mathbb{B}}_p(T)$ with regards to the well-approximation property}
  \label{sec:besapp}
As apparent in \cite{chemincpam}, it turns out that any vector field belonging to the space $\dot{\mathbb{B}}_p(T)$ defined in Definition~\ref{def:bpTp} is indeed well-approximated in the sense of Definition \ref{def:well}.

\begin{prop}
\label{prop:Besovwell}
 Let $T>0$. All $u \in \dot{\mathbb{B}}_p(T)$ are well-approximated on $[0,T]$.
\end{prop}
\begin{proof}

Let us start by checking the first property. Since $S_j = \sum_{j'=-1}^j \Delta_j $, by Bernstein's inequality, there holds
\begin{align*}
    \int_0^T \| \nabla S_j u\|_\infty \dd s \lesssim  \sum_{j'=-1}^j 2^{j'} \int_0^T \| \Delta_{j'} u\|_\infty \dd s \lesssim  \sum_{j'=-1}^j 2^{j'(1+3/p)} \int_0^T \| \Delta_{j'} u\|_p \dd s,
\end{align*}
and we rely on~\eqref{eq:propBp} to conclude. Likewise,
\begin{align*}
    \int_0^t \|  S_j u\|^2_\infty \dd s &\lesssim   \sum_{j'=-1}^j  \int_0^t \| \Delta_{j'} u\|^2_\infty \dd s \lesssim   \sum_{j'=-1}^j  2^{6j/p} \int_0^t \| \Delta_{j'} u\|^2_p \dd s \\
    &\lesssim \sum_{j'=-1}^j  \left[ 2^{j'(-1+3/p)} \| \Delta_{j'} u \|_{\mathrm{L}^\infty(0,T; \mathrm{L}^p)} \right] \left[2^{j'(1+3/p)} \| \Delta_{j'} u \|_{\mathrm{L}^1(0,T; \mathrm{L}^p)} \right],
\end{align*}
and we again use~\eqref{eq:propBp}.
For the second property, we can combine Proposition~\ref{prop:useful} and the commutation estimate from \cite[Lemma 2.2]{chemincpam} to infer:
\begin{lem}
For $u \in {\L}^\infty(0,T;\mathrm{L}^2(\R^3)) \cap \dot{\mathcal{B}}_p(T)$, there holds, for all $j \in \N$,
\begin{equation*}
\| S_j u \otimes S_j u - S_j (u \otimes u) \|_{\mathrm{L}^2(0,T;\mathrm{L}^2)} 
\lesssim 2^{-\frac{j}{p-2}} \| u \|_{{\L}^\infty(0,T;\mathrm{L}^2(\R^3)) \cap \dot{\mathcal{B}}_p(T)}^2. 
\end{equation*}
\end{lem}
 The proof of Proposition~\ref{prop:Besovwell} is therefore complete. $\qedhere$
 
\end{proof}

   \section{Further regularity for Leray solutions and an estimate of $\mathcal{N}_{T,R}(f^\init)$}
\label{sec:leray+}
As recalled in the introduction, for any admissible initial data $(f^\init,u^\init)$ in the sense of Definition~\ref{def:admi}, the VNS system admits at least one weak Leray solution. Just as in \cite{hkm3}, our analysis starts with a simple observation: under a slightly stronger assumption for the kinetic initial data, these Leray solutions actually enjoy finer bounds. More precisely, we have the following result.

           \begin{lem}\label{lem:reg}
Consider an admissible initial condition $(f^\init,u^\init)$ in the sense of Definition~\ref{def:admi} and assume furthermore that 
$M_6 f^\init < \infty$
and that $N_q(f^\init)<\infty$ for some $q>5$. Then, any associated Leray solution $(f,u)$ of the \textup{VNS} system satisfies
\begin{itemize}
             \item[$(i)$] $u\in\Ll^1(\R_+;\mathrm{L}^\infty(\R^3))$, with for all $T>0$,
             $$
             \| u\|_{\mathrm{L}^1(0,T;\mathrm{L}^\infty(\R^3))} \leq C_{T,\E^\init,M_6 f^\init}.
             $$
             \item[$(ii)$] $m_k f\in\Ll^\infty(\R_+;\mathrm{L}^\infty(\R^3))$ for $k\in\{0,1,2\}$, with for all $T>0$,
             $$
             \| m_k \|_{\mathrm{L}^\infty(0,T;\mathrm{L}^\infty(\R^3))} \leq C_{T,\E^\init,M_6 f^\init,N_q(f^\init)}, \qquad q>5.
             $$
             \end{itemize}
           \end{lem}

           \begin{proof}
Let us focus on the property $(i)$ as building on it, we can establish $(ii)$ exactly as in \cite[Proposition 4.6]{hkm3} (see also the proof of Lemma \ref{lem:NR} hereafter). We start with the Duhamel formula
\begin{align}\label{eq:duhamel}  u = e^{t\Delta}u_0 + \int_0^t e^{(t-s)\Delta} \mathbf{P} \big[(u \cdot \nabla) u\big](s)\,\dd s + \int_0^t e^{(t-s)\Delta} \mathbf{P}F(s)\,\dd s,
\end{align}
where $F:=j_f-\rho_f u$ is the Brinkman force. The two first terms of right hand side obviously belong to some $\Ll^1(\R_+;\mathrm{L}^r(\R^3))$ for some exponent $r\in[1,\infty]$ (using the regularity of the heat flow and the fact that $u$ is a Leray solution). Relying on Proposition~\ref{prop:besnegborn} of Appendix~\ref{sec:lp} (to which we refer for the definition of the Besov space hereafter used), controlling these terms in $\Ll^1(\R_+;\dot{\B}_{2,1}^{3/2}(\R^3))$ is sufficient to deduce a bound in $\Ll^1(\R_+;\mathrm{L}^\infty(\R^3))$. This can be done by directly using \cite[Lemme 3.2]{chem}. It remains therefore to treat the last term of \eqref{eq:duhamel}, that is the one involving the Brinkman force $F$. 

\medskip

First, proceeding exactly as in \cite[Lemma 4.7]{ehkm}, we can check that the Brinkman force $F$ actually belongs to $\Ll^2(\R_+;\mathrm{L}^2(\R^3))$. This is the only step in which we crucially use the extra integrability assumption for the sixth moment of the kinetic initial data. Then obviously 
\[I(F) := \int_0^t e^{(t-s)\Delta} \mathbf{P} F(s)\,\dd s,\]
  belongs to $\Ll^1(\R_+;\mathrm{L}^2(\R^3))$ (as the $\mathrm{L}^2(\R^3)$ norm decays along the heat flow). Using the decomposition (see Subsection~\ref{subsec:loc} of the appendix) 
  \begin{align*}
I(F) = \chi \star I(F) +\sum_{j\in\N} \Delta_j I(F),
  \end{align*}
  Hölder's inequality allows to focus on the sum, as \[\|\chi \star I(F)(s)\|_\infty\lesssim \|I(F)(s)\|_{\mathrm{L}^2(\R^3)},\]
  and we have seen that this last upper bound is locally integrable in time. For all $j\in\mathbf{N}$, we write 
  \begin{align*}
    \left\|\Delta_j I(F)(t) \right\|_{\mathrm{L}^2(\R^3)} &\leq \int_0^t \|\Delta_j e^{(t-s)\Delta} \mathbf{P}F(s) \|_{\mathrm{L}^2(\R^3)} \,\dd s \\
                                                                                &\lesssim \int_0^t e^{-c(t-s) 2^{2j}} \|\Delta_j \mathbf{P}F(s)\|_{\mathrm{L}^2(\R^3)}\,\dd s.
  \end{align*}
  where we have applied \cite[Lemma 2.4]{bcd} for the second inequality. 
  Thanks to Young's inequality for convolutions (in time), we thus deduce for all $j\geq 0$
  \begin{align*}
   \int_0^T \left\|\Delta_j I(F)(t) \right\|_{\mathrm{L}^2(\R^3)}\,\dd t\lesssim \frac{1}{2^{2j}} \int_0^T \|\Delta_j \mathbf{P}F(s)\|_{\mathrm{L}^2(\R^3))}\,\dd s,
  \end{align*}
whence, by Cauchy-Schwarz and the continuity of the Leray projector on $\mathrm{L}^2(\R^3)$,
  \begin{align*}
    \int_0^T \sum_{j\in\N} 2^{\frac{3}{2}j} \left\|\Delta_j I(F)(t) \right\|_{\mathrm{L}^2(\R^3)}\,\dd t &\lesssim \int_0^T \sum_{j\geq 0} \frac{1}{2^{\frac12 j}}\|\Delta_j \mathbf{P} F(s)\|_{\mathrm{L}^2(\R^3)}\,\dd s\\
                                                    &\lesssim   \int_0^T \left(\sum_{j\geq 0} \|\Delta_j \mathbf{P}F(s)\|_{\mathrm{L}^2(\R^3)}^2 \right)^{1/2}\,\dd s\\                                                                         &\lesssim  \int_0^T \|F(s)\|_{\mathrm{L}^2(\R^3)}\,\dd s.
  \end{align*}
The conclusion follows just as in the proof of Proposition~\ref{prop:besnegborn}: we use Bernstein inequality (Lemma~\ref{lem:bernstein}) to bound $\|\Delta_j I(F)(t)\|_{\mathrm{L}^\infty(\R^3)}$ by $2^{\frac{3}{2}j}\|\Delta_j I(F)(t)\|_{\mathrm{L}^2(\R^3)}$, and this establishes the normal convergence of the series $\sum_{j\in\N} \Delta_j I(F)$ in $\mathrm{L}^1(0,T;\mathrm{L}^\infty(\R^3))$.
           \end{proof}

           \medskip
           
From now on, we let $\overline{f}\in \mathrm{L}^1(\R^3)\cap \mathrm{L}^\infty(\R^3)$ be 
such that $M_1(|\overline{f}|)$ is finite and $N_q(|\overline{f}|)$ is finite for some $q>4$. We have the followong property.
\begin{lem}
\label{lem:NR}
Recalling $\mathcal{N}_{T,R}$ introduced in Definition~\ref{def:NR},
there holds
\[\mathcal{N}_{T,R}(\overline{f})\lesssim_{T,R} (N_q(|\overline{f}|)+ M_1(|\overline{f}|)+\|\overline{f}\|_{\mathrm{L}^1}).\]
\end{lem}

\begin{proof}
 W.l.o.g. we can assume that $\overline{f}$ is \emph{non-negative} (replacing it by $|\overline{f}|$). Let $R>0$ and $u \in \B_R$. In view of \eqref{eq:ODEs}, a straightforward estimate for the characteristics shows that
 \begin{align}
\label{ineq:Z} |\Z_u(t,0,x,v)- (x+(e^{-t}-1)v,e^{-t}v)|\leq \textnormal{C}_{T,R}.
 \end{align}
 Let's first treat the $\mathrm{L}^\infty(0,T;\mathrm{L}^1(\R^3))$ part of the estimate. For all $t\in[0,T]$, by definition of the push-forward measure we have, using estimate \eqref{ineq:Z} on the third line
  \begin{align*}
    \|m_0|f_u|(t)\|_{\mathrm{L}^1(\R^3)} &+ \|m_1|f_u|(t)\|_{\mathrm{L}^1(\R^3)} = M_0 |f_u|(t) + M_1 |f_u|(t) \\
                                                                &= \iint_{\R^3\times\R^3} \overline{f} + \iint_{\R^3\times\R^3} |\V_u(t,x,v)|\,  \overline{f}(x,v)\,\dd v\,\dd x\\
                                                                &\lesssim_{T,R} \|\overline{f}\|_1+ \iint_{\R^3\times\R^3} (1+|e^{-t}v|)\,\overline{f}(x,v)\,\dd v\,\dd x \\
    &\lesssim_{T,R} \|\overline{f}\|_1 + M_1 \overline{f}.
  \end{align*}
  Now, let's turn to the $\mathrm{L}^\infty(0,T;\mathrm{L}^\infty(\R^3))$ part of the estimate which is a bit more involved. Since  $f_u^t:=f_u(t) =\Z_u(t,0) \# \overline{f}$, we have for any test function $\ffi\in\mathscr{C}^\infty(\R^3)$, 
 \begin{align}
\label{rep1}   \int_{\R^3} \ffi(x) \,m_0 f_u^t(x)\,\dd x= \iint_{\R^3\times\R^3} \ffi(\X_u(t,0,x,v))\,\overline{f}(x,v)\,\dd v\,\dd x,\\
\label{rep2}\int_{\R^3} \ffi(x) \,m_1 f_u^t\,\dd x  = \iint_{\R^3\times\R^3} \ffi(\X_u(t,0,x,v))|\V_u(t,0,x,v)|\,\overline{f}(x,v)\,\dd v\,\dd x.
  \end{align}
  We infer from the above representation formulas \eqref{rep1} -- \eqref{rep2} for any test function $\ffi\in\mathscr{D}(\R^3)$
  \begin{align*}
    \Big|\int_{\R^3} \ffi(x)\,(m_0 f_u^t&+m_1 f_u^t)(x)\,\dd x\Big|\\
    &\leq N_q(\overline{f})\iint_{\R^3\times\R^3} \frac{|\ffi|(\X_u(t,0,x,v))}{1+|v|^q} (1+|\V_u(0,t,x,v)|)\,\dd x\,\dd v \\
    &= e^{3t} N_q(\overline{f})\iint_{\R^3\times\R^3}\frac{|\ffi|(x)}{1+|\V_u(t,0,x,v)|^q} (1+|v|)\,\dd x\,\dd v\\
                    &\lesssim_{T} N_q(\overline{f})\|\ffi\|_1  + N_q(\overline{f}) \iint_{\R^3\times\R^3} \frac{|\ffi|(x)}{1+|\V_u(0,t,x,v)|^q} |v|\,\dd x \,\dd v.
  \end{align*}
  Now, in the last integral we split the velocity phase component as $|e^t v|\gtrless \textnormal{C}_{T,R}$ to write for all $t$ and $x$, using \eqref{ineq:Z}
  \begin{align*}
    \int_{\R^3} \frac{1}{1+|\V_u(0,t,x,v)|^q} |v| \,\dd v &\lesssim_{T,R} 1+\int_{|e^t v|\geq \textnormal{C}_{T,R}} \frac{1}{1+|\V_u(0,t,x,v)|^q} |v|\,\dd v\\
    &\lesssim_{T,R} 1+\int_{|e^t v| \geq \textnormal{C}_{T,R}} \frac{1}{1+|e^t v|^q-\textnormal{C}_{T,R}} |v|\,\dd v
   <\infty,
  \end{align*}
 where $q>4$ is crucially used to get the finiteness of this last integral. Gathering all these estimates we have therefore for any test function $\ffi$ 
  \begin{align*}
    \left|\int_{\R^3} \ffi(x) \,(m_0 f_u^t+m_1 f_u^t)(x)\,\dd x \right| \lesssim_{T,R} \|\ffi\|_1 N_q(\overline{f}),
  \end{align*}
  which proves by duality
  \begin{align*}
\sup_{t\in[0,T]} \Big\{ \|m_0 f_u^t\|_{\mathrm{L}^\infty(\R^3)} + \|m_1 f_u^t\|_{\mathrm{L}^\infty(\R^3)} \Big\} \lesssim_{T,R} N_q(\overline{f}). \qquad \qedhere
  \end{align*}
\end{proof}


\section{Controlling the Wasserstein distance with the functional $Q_{\Z_1,\Z_2}$}\label{sec:contwas}

Here we relate the functional defined in \eqref{def:Q} with the Wasserstein distance $\W_1$ between the corresponding push-forward measures by the associated flows. More precisely, we consider two finite measures $f_1^\init$ and $f_2^\init$ on $\R^3\times \R^3$ sharing the same total mass and having finite two first moments. We fix $T>0$ and consider for $k\in\{1,2\}$, maps $\Z_k:[0,T]\times \R^3\times \R^3\to \R^3\times \R^3$, such that  for all $t\in [0,T]$, $\Z_k(t,\cdot,\cdot)$ is invertible on $\R^3\times \R^3$ and such that $\Z_k(0,x,v)=(x,v)$, for all $(x,v)$. Moreover, we assume that $t\mapsto \Z_1(t)=(\X_1(t),\V_1(t))$ is the unique Lipschitz solution to the ODE : 
\begin{equation}\label{eq:EDO-appendix}
    \begin{cases}
     \displaystyle \frac{\dd}{\dd t}{\X}_1(t,x,v)=\V_1(t,x,v),\quad \X_1(0,x,v)=x,\\
   \displaystyle   \frac{\dd}{\dd t}{\V}_1(t,x,v)=u(t,\X_1(t,x,v))-\V_1(t,x,v),\quad \V_1(0,x,v)=v,
    \end{cases}
\end{equation}
for some vector field $u\in \mathrm{L}^1([0,T];\W^{1,\infty}(\R^3))$ (see \eqref{eq:ODEs}). In particular, $\Z_1$ belongs to $\mathscr{C}^0([0,T];\mathscr{C}^{0,1}(\R^3\times\R^3))$. 
\begin{prop}\label{prop:norm-W}
For $k=1,2$, consider two finite measures $f_k^\init$ sharing the same total mass and finite two first moments. We set $f_k(t)=\Z_k(t)\#f_k^\init$ and define
$$Q_{\Z_1,\Z_2}(t)=\iint_{\R^3\times\R^3}|\Z_1(t,x,v)-\Z_2(t,x,v)|^2 \,\dd f_2^\init(x,v).$$
Then we have for $t\in[0,T]$
\begin{align*}
    \W_1(f_1(t),f_2(t))^2 \lesssim \W_1(f_1^\init,f_2^\init)^2 + Q_{\Z_1,\Z_2}(t),
\end{align*}
where the symbol $\lesssim$ depends only the norm of $u_1$ in $\mathrm{L}^1([0,T];\W^{1,\infty}(\R^3))$ and the shared mass of $f_1^\init$ and $f_2^\init$.
\end{prop}
\begin{proof}
Recall the Kantorovitch duality, for two measures $\mu$ and $\nu$ on $\R^3\times\R^3$ of same mass
    \begin{align*}
            \W_1(\mu,\nu)
            =\sup \left\{ \iint_{\R^3\times\R^3} \varphi \,\dd(\mu-\nu)(x,v),\quad \varphi\in \mathscr{C}^{0,1}(\R^3\times\R^3),\quad \|\nabla_{x,v}\varphi\|_{\infty}\leq 1 \right\}.
    \end{align*}
For a given test function $\varphi\in \mathscr{C}^{0,1}(\R^3\times\R^3)$ such that  $\|\nabla_{x,v}\varphi\|_{\infty}\leq 1$, we have for $t\in[0,T]$ 
    \begin{align*}
        \iint_{\R^3\times\R^3} \ffi &\,\dd(f_1(t)-f_2(t))(x,v)   \\&=\iint_{\R^3\times\R^3} \varphi(\Z_1(t,x,v))\dd f_1^\init(x,v)-\iint_{\R^3\times\R^3}\varphi(\Z_2(t,x,v))\, \dd f_2^\init(x,v)\\
                                                      &=\iint_{\R^3\times\R^3} \varphi(\Z_1(t,x,v))\,\dd (f_1^\init-f_2^\init)(x,v) 
                                                   \\
                                                   &\qquad\qquad+ \iint_{\R^3\times\R^3} \left[\varphi(\Z_1(t,x,v))-\varphi(\Z_2(t,x,v))\right]\,\dd f_2^\init(x,v)\\
      &= I_1+I_2.
    \end{align*}
    To estimate $I_1$, we simply note that by \eqref{eq:EDO-appendix}, for all $t\in[0,T]$, the map $\Psi : (x,v) \mapsto\ffi(\Z_1(t,x,v))$ satisfies
    \begin{equation*}
        \|\nabla_{x,v}\Psi\|_{\infty}\leq
        \|\nabla_{x,v}Z_1(t,\cdot,\cdot)\|_{\infty}\|\nabla_{x,v}\varphi\|_{\infty}\leq C\left(T, \|u_1\|_{\mathrm{L}^1([0,T];\W^{1,\infty})}\right).
    \end{equation*}
    
    We have therefore directly $I_1 \lesssim \W_1(f^\init_1,f^\init_2)$. For $I_2$, using that $\ffi$ is lipschitz with a constant less than $1$ together with the fact that $f_2^\init$ is a finite measure to write
   \begin{align*}
       I_2&\leq \iint_{\R^3\times\R^3} |\Z_2(t,x,v)-\Z_1(t,x,v)|\,\dd f_2^\init(x,v)\\
            &\lesssim Q_{\Z_1,\Z_2}(t)^{1/2}.
   \end{align*}
The conclusion follows. $\qedhere$
\end{proof}

  \bibliographystyle{plain}
\bibliography{biblio}

    \end{document}